\documentclass[10pt]{article}

\usepackage[authoryear,round]{natbib}                   
\usepackage{fixltx2e,amsmath}                                                       
\usepackage{amsfonts}                                                       
\usepackage{graphicx}                                                       
\usepackage{subfigure}                                                      
\usepackage{amssymb}                                                        
\usepackage[mathscr]{eucal}                                             
\usepackage{cancel}                                                             
\usepackage[normalem]{ulem}                                                                 
\usepackage{pstricks}
\usepackage{rotating}
\usepackage{mdwlist}                                                            
\usepackage[paperwidth=8.5in,paperheight=11in,top=1.25in, bottom=1.25in, left=1.00in, right=1.00in]{geometry}
\usepackage{mathtools}                                                      
\mathtoolsset{showonlyrefs=true}                                    
\usepackage{amsthm}                                                             
\theoremstyle{plain}
\newtheorem{theorem}{Theorem}

\newtheorem{proposition}{Proposition}
\newtheorem{corollary}{Corollary}
\theoremstyle{definition}

\newtheorem{example}{Example}
\newtheorem{remark}{Remark}

\MakeRobust{\eqref}                                                                     

\usepackage{hyperref}

\newcommand{\G}{\Gamma}

\def \R  {{\mathbb {R}}}

\def \x {{\xi}}
\def \g {{\gamma}}

\def \t {{\tau}}

\def \m {{\mu}}

\def \p {{\partial}}

\def \F {{\mathscr{F}}}

\newcommand{\<}{\langle}
\renewcommand{\>}{\rangle}
\renewcommand{\(}{\left(}
\renewcommand{\)}{\right)}
\renewcommand{\[}{\left[}
\renewcommand{\]}{\right]}

\newcommand\Eb{\mathbb{E}}
\newcommand\Pb{\mathbb{Q}}
\newcommand\Rb{\mathbb{R}}
\newcommand\Ib{\mathbb{I}}

\newcommand\Ac{\mathscr{A}}

\newcommand\Dc{\mathscr{D}}
\newcommand\Ec{\mathscr{E}}
\newcommand\Fc{\mathscr{F}}
\newcommand\Gc{\mathscr{G}}

\newcommand\Oc{\mathscr{O}}

\newcommand\Sc{\mathscr{S}}
\newcommand\Nc{\mathscr{N}}

\newcommand\mf{\mathfrak{m}}

\newcommand\Om{\Omega}
\newcommand\sig{\sigma}

\newcommand\gam{\gamma}
\newcommand\Gam{\Gamma}
\newcommand\lam{\lambda}
\newcommand\del{\delta}

\newcommand\nub{\bar{\nu}}

\newcommand\xb{\bar{x}}

\newcommand\uh{\widehat{u}}
\newcommand\vh{\widehat{v}}

\newcommand\hh{\widehat{h}}

\renewcommand\d{\partial}

\begin{document}

\title{
A family of density expansions for L\'evy-type processes}

\author{
Matthew Lorig
\thanks{ORFE Department, Princeton University, Princeton, USA.  Work partially supported by NSF grant DMS-0739195}
\and
Stefano Pagliarani
\thanks{Dipartimento di Matematica, Universit\`a di Padova, Padova, Italy}
\and
Andrea Pascucci
\thanks{Dipartimento di Matematica,
Universit\`a di Bologna, Bologna, Italy}
}

\date{\today}

\maketitle

\centerline{\it To the memory of our dear friend and esteemed colleague Peter Laurence.}

\begin{abstract}
We consider a defaultable asset whose risk-neutral pricing dynamics are described by an
exponential L\'evy-type martingale subject to default.  This class of models allows for local
volatility, local default intensity, and a locally dependent L\'evy measure. Generalizing and
extending the novel adjoint expansion technique of \citet*{RigaPagliaraniPascucci}, we derive
{a family of asymptotic expansions} for the transition density of the underlying as well as
for European-style option prices and defaultable bond prices.  For the density
expansion, we also provide error bounds for the truncated asymptotic series.
Our method is numerically efficient; approximate transition densities and
European option prices are computed
via Fourier transforms;
approximate bond prices are computed as finite series. Additionally, as in
\cite{RigaPagliaraniPascucci}, for models with Gaussian-type jumps, approximate option prices can
be computed in closed form.
Sample Mathematica code is provided.
\end{abstract}


%
%

\section{Introduction and literature review}
\label{sec:intro}

A \emph{local volatility} model is a model in which the volatility $\sig_t$ of an asset $X$ is a
function of time $t$ and the present level of $X$.  That is, $\sig_t = \sig(t,X_t)$.  Among local
volatility models, perhaps the most well-known is the constant elasticity of variance (CEV) model
of \citet*{CoxCEV}.  One advantage of local volatility models is that transition densities of the
underlying -- as well as European option prices -- are often available in closed-form as infinite
series of special functions (see \citet*{linetskybook} and references therein).  Another advantage
of local volatility models is that, for models whose transition density is not available in closed
form, accurate density and option price approximations are readily available (see,
\cite{pagliarani2011analytical}, for example). Finally, \citet*{dupire1994pricing} shows that one
can always find a local volatility function $\sig(t,x)$ that fits the market's implied volatility
surface exactly.  Thus, local volatility models are quite flexible.
\par
Despite the above advantages, local volatility models do suffer some shortcomings.  Most notably, local volatility models do not allow for the underlying to experience jumps, the need for which is well-documented in literature (see \citet*{eraker} and references therein).  Recently, there has been much interest in combining local volatility models and models with jumps.  \citet*{andersen2000jump}, for
example, discuss extensions of the implied diffusion approach of \citet*{dupire1994pricing} to
asset processes with Poisson jumps (i.e., jumps with finite activity).  And \citet*{gobet-smart}
derive analytically tractable option pricing approximations for models that include local volatility and a Poisson jump process.   Their approach relies on asymptotic expansions around small diffusion and
small jump frequency/size limits.   More recently, \citet*{RigaPagliaraniPascucci} consider
general local volatility models with independent L\'evy jumps (possibly infinite activity).
Unlike, \citet{gobet-smart}, \citet{RigaPagliaraniPascucci} make no small jump intensity/size
assumption. Rather the authors construct an approximated solution by expanding the local volatility function as
a power series.  While all of the methods described in this paragraph allow for local volatility
and \emph{independent} jumps, none of these methods allow for \emph{state-dependent} jumps.
\par
Stochastic jump-intensity was recently identified as an important feature of equity models (see
\citet*{christoffersen2009}).  A locally dependent L\'evy measure allows for this possibility.
Recently, two different approaches have been taken to modeling assets with locally-dependent jump
measures.  \citet*{carr} time-change a local volatility model with a L\'evy subordinator.  In
addition to admitting exact option-pricing formulas, the subordination technique results in a
locally-dependent L\'evy measure.  \citet*{lorigCEVLevy} considers another class of models that
allow for state-dependent jumps.  The author builds a L\'evy-type processes with local volatility,
local default intensity, and a local L\'evy measure by considering state-dependent perturbations
around a constant coefficient L\'evy process.  In addition to pricing formula, the author provides
an exact expansion for the induced implied volatility surface.
\par
In this paper, we consider scalar L\'evy-type processes with regular coefficients, {which
naturally {include} all the models mentioned above}.  Generalizing and extending the
methods of \citet{RigaPagliaraniPascucci}, we derive {a family of asymptotic expansions} for the
transition densities of these processes, as well as for European-style derivative prices and
defaultable bond prices.  The key contributions of this manuscript are as follows:
\begin{itemize*}
\item We allow for a locally-dependent L\'evy measure and local default intensity, whereas \citet{RigaPagliaraniPascucci} consider a locally \emph{independent} L\'evy measure and do not allow for the possibility of default.  A state-dependent L\'evy measure is an important feature because it allows for incorporating local dependence into infinite activity L\'evy models that have no diffusion component, such as Variance Gamma {(\citet*{madan1998variance})} and CGMY/Kobol {(\citet*{levendorskiibook,CGMY})}.
\item Unlike \citet{gobet-smart}, we make no small diffusion or small jump size/intensity assumption.  Our formulae are valid for any L\'evy type process with smooth and bounded coefficients, independent of the relative size of the coefficients.
\item Whereas \citet{RigaPagliaraniPascucci} expand the local volatility and drift functions as a Taylor series about an arbitrary point, i.e. $f(x) = \sum_n a_n (x-\xb)^n $, in order to achieve their approximation result, we expand the local volatility, drift, killing rate and L\'evy measure in an arbitrary basis, i.e.  $f(x) = \sum_n c_n B_n(x)$.
This is advantageous because the Taylor series typically converges only locally, whereas other choices of the basis functions $B_n$ may provide global convergence in suitable functional spaces.
\item Using techniques from pseudo-differential calculus, we provide {explicit formulae for the Fourier transform of} every term in the transition density and option-pricing expansions. In the case of state dependent Gaussian jumps the respective inverse Fourier transforms can be explicitly computed,
thus providing closed form approximations for densities and prices. In the general case, the density and pricing approximations can be computed quickly and easily as inverse Fourier transforms.
Additionally, when considering defaultable bonds, approximate prices are computed as a finite sum; no numerical integration is required even in the general case.
\item For models with Gaussian-type jumps, we provide pointwise error estimates for transition densities.  Thus, we extend the previous results of \cite{RigaPagliaraniPascucci} where only the purely diffusive case is considered.  Additionally, our error estimates allow for jumps with locally dependent mean, variance and intensity.  Thus, for models with Gaussian-type jumps, our results also extend the results of \cite{gobet-smart}, where only the case of a constant L\'evy measure is considered.
\end{itemize*}
\par
The rest of this paper proceeds as follows. In Section, \ref{sec:model} we introduce a general
class of exponential L\'evy-type models with locally-dependent volatility, default intensity and
L\'evy measure.  We also describe our modeling assumptions. Next, in Section \ref{sec:pricing}, we
introduce the European option-pricing problem and derive a partial integro-differential equation
(PIDE) for the price of an option. In Section \ref{sec:formal} we derive a formal asymptotic
expansion (in fact, a family of asymptotic expansions) for the function that solves the option
pricing PIDE (Theorem \ref{thm:v}). Next, in Section \ref{errbou}, we provide rigorous error
estimates for our asymptotic expansion for models with Gaussian-type jumps (Theorem \ref{t1}).
Lastly, in Section \ref{sec:examples}, we provide numerical
examples that illustrate the effectiveness and versatility of our methods.
Technical
proofs are provided in the Appendix. Some concluding remarks are given in Section
\ref{sec:conclusion}.
\par
We mention specifically that the arguments needed to provide rigorous error estimates for our asymptotic expansions are quite extensive.  As such, in this manuscript, we provide only an outline of the proof of Theorem \ref{t1}.  The full proof of Theorem \ref{t1}, as well as further numerical examples, can be found in a companion paper \citet*{lpp-1.5}.

%
%

\section{General L\'evy-type exponential martingales}
\label{sec:model}

For simplicity, we assume a frictionless market, no arbitrage, zero interest rates and no
dividends.  Our results can easily be extended to include locally dependent interest rates and
dividends.  We take, as given, an equivalent martingale measure $\Pb$, chosen by the market on a
complete filtered probability space \mbox{$(\Om,\Fc,\{\Fc_t,t\geq0\},\Pb)$} satisfying the usual
hypothesis of completeness and right continuity.  The filtration $\Fc_t$ represents the history of
the market.  All stochastic processes defined below live on this probability space and all
expectations are taken with respect to $\Pb$.  We consider a defaultable asset $S$ whose
risk-neutral dynamics are given by
\begin{align}
\left.
\begin{aligned}
S_t
    &=  \Ib_{\{ \zeta > t\}} e^{X_t}, \\
dX_t
    &=  \mu(t,X_t) dt + \sig(t,X_t) dW_t + \int_\Rb d\overline{N}_t(t,X_{t-},dz) z ,  \\
d\overline{N}_t(t,X_{t-},dz)
    &=  dN_t(t,X_{t-},dz) - \nu(t,X_{t-},dz) dt, \\
\zeta
    &=  \inf \left\{ t \geq 0 : \int_0^t \gam(s,X_s) ds \geq \Ec \right\}
\end{aligned}
\right\} \label{eq:dX}
\end{align}
Here, $X$ is a L\'evy-type process with local drift function $\mu(t,x)$, local volatility function
$\sig(t,x) \geq 0$ and state-dependent L\'evy measure $\nu(t,x,dz)$.  We shall denote by $\Fc_t^X$
the filtration generated by $X$.  The random variable $\Ec \sim \text{Exp}(1)$ has an exponential
distribution and is independent of $X$.  Note that $\zeta$, which represents the default time of
$S$, is constructed here trough the so-called \textit{canonical construction} (see
\cite{bielecki2001credit}), and is the first arrival time of a doubly stochastic Poisson process
with local intensity function $\gam(t,x) \geq 0$.  This way of modeling default is {also}
considered in a local volatility setting in \citet*{JDCEV,linetsky2006bankruptcy}, and for
exponential L\'evy models in \cite{capponi}.
\par
We assume that the coefficients are measurable in $t$ and suitably smooth in $x$ to
ensure the existence of a solution to \eqref{eq:dX}  (see \citet*{oksendal2}, Theorem 1.19). We
also assume the following boundedness condition which is rather standard in the financial
applications: there exists a L\'evy measure
\begin{align}
  \nub(dz)    &:= \sup_{(t,x) \in \Rb^+ \times \Rb} \nu(t,x,dz)
\end{align}
such that
\begin{align}
\int_\Rb \ \nub(dz) \min(1,z^2)
        &<      \infty, &
\int_{|z| \geq 1}  \nub(dz) e^z
        &<      \infty, &
\int_{|z| \geq 1}  \nub(dz) |z|
        &<      \infty. \label{eq:conditions}
\end{align}

Since $\zeta$ is not $\Fc_t^X$-measurable we introduce {the filtration
$\Fc_t^D=\sigma\left(\{\zeta\le s\right), s\le t\}$} in order to keep track of the event $\{\zeta
\leq t \}$.
The filtration of a market observer, then, is $\Fc_t = \Fc_t^X \vee \Fc_t^D$. In the absence of
arbitrage, $S$ must be an {$\Fc_t$}-martingale. Thus, the drift $\mu(t,x)$ is fixed by
$\sig(t,x)$, $\nu(t,x,dz)$ and $\gam(t,x)$ in order to satisfy the martingale condition\footnote{
We provide a derivation of the martingale condition in Section \ref{sec:pricing} Remark \ref{rr11}
below. }
\begin{align}
\mu(t,x)
    &=  \gam(t,x) - a(t,x) - \int_\Rb \nu(t,x,dz) (e^z-1-z), &
a(t,x)
    &:= \frac{1}{2}\sig^2(t,x). \label{eq:drift}
\end{align}

We remark that the existence of the density of $X$ is not strictly necessary in our analysis.
Indeed, since our formulae are carried out in Fourier space, we provide approximations of the
characteristic function of $X$ and all of our computations are still formally correct even when
dealing with distributions that are not absolutely continuous with respect to the Lebesgue
measure.

%
%

\section{Option pricing}
\label{sec:pricing}

We consider a European derivative expiring at time $T$ with payoff $H(S_T)$ and we denote by $V$
its  no-arbitrage price. For convenience, we introduce
\begin{align}
h(x)
    &:= H(e^x)\qquad\text{and}\qquad K:= H(0).
\end{align}
\begin{proposition}\label{p1}
The price $V_t$ is given by
\begin{align}\label{e1}
 V_{t}
    &=      K+\Ib_{\{\zeta>t\}} \Eb \[e^{-\int_t^T \gam(s,X_s) ds} \left(h(X_T)-K\right)  | X_t \], &
 t
        &\le    T.
\end{align}
\end{proposition}
\noindent
The proof can be found in Section 2.2 of \citet{linetsky2006bankruptcy}. Because our notation
differs from that of \citet{linetsky2006bankruptcy}, and because a short proof is possible by using
the results of \citet*{yorbook},
for the reader's convenience, we provide a derivation of Proposition~\ref{p1} here.
\begin{proof}
Using risk-neutral pricing, the value $V_t$ of the derivative at time $t$ is given by the
conditional expectation of the option payoff
\begin{align}
V_t
    &=  \Eb \[ H(S_T) | \Fc_t \] \\
    &=  \Eb \[ h(X_T) \Ib_{\{\zeta>T\}} | \Fc_t \] + K  \Eb \[ \Ib_{\{\zeta \leq T\}}| \Fc_t \] \\
    &=  \Eb \[ h(X_T) \Ib_{\{\zeta>T\}} | \Fc_t \] + K - K  \Eb \[ \Ib_{\{\zeta > T\}}| \Fc_t \] \\
    &=  K+\Ib_{\{\zeta>t\}} \Eb \[e^{-\int_t^T \gam(s,X_s) ds}\left(h(X_T)-K\right)| \Fc_t^X \] \\
    &=  K+\Ib_{\{\zeta>t\}} \Eb \[e^{-\int_t^T \gam(s,X_s) ds} \left(h(X_T)-K\right)  | X_t \],
\end{align}
where we have used Corollary 7.3.4.2 from \citet*{yorbook} to write
\begin{align}
\Eb\[  (h(X_T)-K)\Ib_{\{\zeta>T\}} | \Fc_t \]
    &=  \Ib_{\{ \zeta> t\}} \Eb \[  (h(X_T)-K) e^{-\int_t^T \gam(s,X_s) ds} | \Fc_t^X \] .
\end{align}
\end{proof}
\begin{remark}\label{rr11} By Proposition \ref{p1} with $K=0$ and $h(x)=e^{x}$, we have
that the martingale condition $S_{t}=\Eb\[S_{T}|\Fc_{t}\]$
is equivalent to
\begin{align}
\Ib_{\{\zeta>t\}}e^{X_{t}}=\Ib_{\{\zeta>t\}}\Eb\[e^{-\int_{t}^{T}\gamma\left(s,X_{s}\right)
   ds+X_{T}}|\Fc_{t}\].
\end{align}
Therefore, we see that $S$ is a martingale if and only if the process $\exp\(-\int_{0}^{t}\gamma\left(s,X_{s}\right)ds+X_{t} \)$
is a martingale.  The drift condition \eqref{eq:drift} follows by applying the It\^o's formula to the process $\exp\(-\int_{0}^{t}\gamma\left(s,X_{s}\right)ds+X_{t} \)$ and setting the drift term to zero.
\end{remark}
\noindent
From \eqref{e1} one sees that, in order to compute the price of an option, we must evaluate
functions of the form\footnote{Note: we can accommodate stochastic interest rates and dividends of
the form $r_t = r(t,X_t)$ and $q_t=q(t,X_t)$ by simply making the change: $\gam(t,x) \to \gam(t,x)
+ r(t,x)$ and $\mu(t,x) \to \mu(t,X_t) + r(t,X_t) - q(t,X_t)$.}
\begin{align}\label{expectation}
v(t,x)
    &:= \Eb \[e^{-\int_t^T \gam(s,X_s) ds}h(X_T)| X_t = x \] .
\end{align}
By a direct application of the Feynman-Kac representation theorem,
see for instance \cite[Theorem 14.50]{Pascucci2011},
the classical solution of the following Cauchy problem,
\begin{align}
(\d_t + \Ac^{(t)}) v
    &=  0, &
v(T,x)
    &=  h(x), \label{eq:v.pide}
\end{align}
when it exists, is equal to the function $v(t,x)$ in \eqref{expectation}, where
\begin{align}
\Ac^{(t)} f(x) &=  \gam(t,x) ( \d_{x}f(x) - f(x) ) + a(t,x) ( \d_{x}^2 f(x) - \d_{x}f(x) )\\
           &\quad - \int_\Rb \nu(t,x,dz)(e^z-1-z)  \d_{x}f(x) + \int_\Rb \nu(t,x,dz)( f(x+z) - f(x) - z \d_{x}f(x) ) , \label{eq:A}
\end{align}
is the characteristic operator of the SDE \eqref{eq:dX}.  In order to shorten the notation, in the sequel we will suppress the explicit dependence on $t$ in $\Ac^{(t)}$ by referring to it just as $\Ac$.
\par
Sufficient conditions for the existence and uniqueness of solutions of second order elliptic
integro-differential equations are given in Theorem II.3.1 of \citet*{GarroniMenaldi}. We denote
by $p(t,x;T,y)$ the fundamental solution of the operator $(\d_t + \Ac)$, which is defined as the
solution of $\eqref{eq:v.pide}$ with $h=\del_y$. Note that $p(t,x;T,y)$ represents also the
transition density of $\log S$ \footnote{Here with $\log S$ we denote the process $X_t\Ib_{\{
\zeta > t\}} -\infty\,\Ib_{\{ \zeta \leq t\}}$.  }
\begin{align}
 p(t,x;T,y )dy =  \Pb [ \log S_T \in dy | \log S_t = x ], \qquad  x,y\in\Rb, \qquad t < T.
\end{align}
Note also that $p(t,x;T,y)$ is not a probability density since (due to the possibility that $S_T=0$) we have
\begin{align}
\int_\Rb p(t,x;T,y) dy
    \leq 1.
\end{align}
Given the existence of the fundamental solution of $(\d_t + \Ac)$, we have that for any $h$ that
is integrable with respect to the density $p(t,x;T,\cdot)$, the Cauchy problem \eqref{eq:v.pide}
has a classical solution that can be represented as
\begin{align}
v(t,x)
    &= \int_\Rb h(y) p(t,x;T,y) dy. \label{eq:v.def1}
\end{align}
\begin{remark}
If $\Gc$ is the generator of a scalar Markov process and $\text{dom}(\Gc)$ contains $\Sc(\Rb)$,
the Schwartz space of rapidly decaying functions on $\Rb$, then $\Gc$ must have the following
form:
\begin{align}
\Gc f(x)
    &=  -\gam(x) f(x) + \mu(x) \d_{x}f(x) + a(x) \d_{x}^2 f(x)+ \int_\Rb \nu(x,dz) ( f(x+z) - f(x) - \Ib_{\{|z|<R\}}z \d_{x}f(x) ), \label{eq:G}
\end{align}
where $\gam \geq 0$, $a \geq 0$, $\nu$ is a L\'evy measure for every $x$ and $R \in [0,\infty]$
(see \cite*{hoh1998pseudo}, Proposition 2.10).  If one enforces on $\Gc$ the drift and
integrability conditions \eqref{eq:conditions} and \eqref{eq:drift}, which are needed to ensure
that $S$ is a martingale, and allow setting $R=\infty$, then the operators \eqref{eq:A} and
\eqref{eq:G} coincide (in the time-homogeneous case).  Thus, the class of models we consider in
this paper encompasses \emph{all} non-negative scalar Markov martingales that satisfy the
regularity and boundedness conditions of Section \ref{sec:model}.
\end{remark}
\begin{remark}
In what follows we shall systematically make use of the language of pseudo-differential calculus.
More precisely, let us denote by
\begin{align}
\psi_\x(x) = { \psi_x(\x) }
    &=  \frac{1}{\sqrt{2\pi}}\,e^{i \x x},\qquad x,\x\in\R , \label{e5}
\end{align}
the so-called \emph{oscillating exponential function}. Then $\Ac$ can be characterized by its action on
oscillating exponential functions.  Indeed, we have
\begin{align}
\Ac \psi_\x(x)
    &=  \phi(t,x,\x ) \psi_\x(x) ,
\end{align}
where
\begin{align}\label{e6}
\phi(t,x,\x )
    &=\gam(t,x) ( i\x  - 1 ) + a(t,x) ( -\x ^2 - i \x  )\\
    &\quad- \int_\Rb \nu(t,x,dz) (e^z-1-z) i \x
     +  \int_\Rb \nu(t,x,dz) ( e^{i \x  z} - 1 - i \x  z ) ,
\end{align}
is called the \emph{symbol} of $\Ac$. Noting that
\begin{align}
 e^{z \d_{x}} u(x)
    &=  \sum_{n=0}^\infty \frac{z^n}{n!} \d_{x}^n u(x)= u(x+z),
\end{align}
for any analytic function $u(x)$, we have
\begin{align}\label{aan1}
     \int_\Rb \nu(t,x,dz)\( u(x+z) - u(x) - z \d_{x} u(x) \)= \int_\Rb \nu(t,x,dz) \left(e^{z \d_{x}} - 1 - z \d_{x}\right) u(x).
\end{align}
Then $\Ac$ can be represented as
\begin{align}
 \Ac
    =  \phi(t,x,\Dc), \qquad
 \Dc    =-i \p_{x},
\end{align}
since by \eqref{e6} and \eqref{aan1}
\begin{align}
\phi(t,x,\Dc)
    &= \gam(t,x) (\p_{x} - 1 ) + a(t,x) (\p_{x}^2 - \p_{x}) \\ &
       - \int_\Rb \nu(t,x,dz) (e^z-1-z) \p_{x}
            + \int_\Rb \nu(t,x,dz) \left(e^{z \p_{x}} - 1 - z \p_{x} \right). \label{eq:Abis}
\end{align}
If coefficients $a(t), \gam(t), \nu(t,dz)$ are independent of $x$, then we have the usual
characterization of $\Ac$ as a multiplication by $\phi$ operator in the Fourier space:
  $$\Ac =\F^{-1}\left(\phi(t,\cdot)\F\right), \qquad\qquad \phi(t,\cdot):=\phi(t,x,\cdot),$$
where $\F$ and $\F^{-1}$ denote the (direct) Fourier and inverse Fourier transform operators respectively:
\begin{align}
\F f(\x)=\hat{f}(\x)
    &:= \frac{1}{\sqrt{2\pi}}\int_\Rb e^{-i \x x} f(x)dx , &
\F^{-1} f(x)
    &= \frac{1}{\sqrt{2\pi}}\int_\Rb e^{i \x x} f(\x)d\x .
\end{align}
Moreover, if the coefficients $a, \gam, \nu(dz)$ are independent of both $t$ and $x$, then $\Ac$ is the
generator of a L\'evy process $X$ and $\phi(\cdot):=\phi(t,x,\cdot)$ is the characteristic exponent of $X$:
\begin{align}
\Eb \left[e^{i\x X_{t}}\right]
    &=  e^{t\phi(\x)}.
\end{align}
\end{remark}

%
%

\section{Density and option price expansions {(a formal description)}}
\label{sec:formal}
Our goal is to construct an approximate solution of Cauchy problem \eqref{eq:v.pide}. We assume
that the symbol of $\Ac$ admits an expansion of the form
\begin{align}\label{e10}
 \phi(t,x,\x) &=  \sum_{n=0}^\infty B_{n}(x)\phi_n(t,\x),
\end{align}
where $\phi_n(t,\x)$ is of the form
\begin{align}
\phi_{n}(t,\x)
    &=  \gam_n(t) ( i\x - 1 ) + a_n(t) (- \x^2 - i\x ) \\ &\qquad
       - \int_\Rb \nu_n(t,dz) (e^z-1-z) i\x + \int_\Rb \nu_n(t,dz) ( e^{iz \x} - 1 - iz \x ) .                              \label{e8}
\end{align}
and $\{B_{n}\}_{n\ge 0}$ is some expansion basis with $B_n$ being an analytic function for each $n
\geq 0$, and  $B_{0}\equiv1$ (see Examples \ref{ex1}, \ref{ex:two-pt} and \ref{ex:L2} below). Note
that $\phi_n(t,\x)$ is the symbol of an operator
\begin{align}\label{e41}
\Ac_{n}
    &:=\phi_n(t,\Dc), \qquad \Dc=-i \p_{x} ,
\end{align}
so that
\begin{align}
\Ac_{n} \psi_\xi(x)
    &=\phi_n(t,\x) \psi_\xi(x) . \label{e13}
\end{align}
Thus, formally the generator $\Ac$ can be written as follows
\begin{align}\label{eq:A.expand}
 \Ac=
    \sum_{n=0}^\infty B_{n}(x)\Ac_{n}.
\end{align}
Note that $\Ac_{0}$ is the generator of a time-dependent L\'evy-type process $X^{(0)}$. In the
time-independent case $X^{(0)}$ is a L\'evy process and $\phi_0(\cdot):=\phi_0(t,\cdot)$ is its
characteristic exponent.
\begin{example}[Taylor series expansion]\label{ex1}
\citet*{RigaPagliaraniPascucci} approximate the drift and diffusion coefficients of $\Ac$ as a power series about an arbitrary point $\xb \in \Rb$.  In our more general setting, this corresponds to setting $B_{n}(x)=(x-\xb)^{n}$ and expanding the diffusion and killing coefficients $a(t,\cdot)$ and $\gam(t,\cdot)$, as well as the L\'evy measure $\nu(t,\cdot,dz)$ as follows:
\begin{align}
\left. \begin{aligned}
a(t,x)
    &=  \sum_{n=0}^\infty a_n(t,\xb)B_{n}(x), &
a_n(t,\xb)
    &=  \frac{1}{n!}\d_x^n a(t,\xb), \\
\gam(t,x)
    &=  \sum_{n=0}^\infty \gam_n(t,\xb)B_{n}(x), &
\gam_n(t,\xb)
    &=  \frac{1}{n!}\d_x^n \gam(t,\xb), \\
\nu(t,x,dz)
    &=  \sum_{n=0}^\infty \nu_n(t,\xb,dz)B_{n}(x), &
\nu_n(t,\xb,dz)
    &=  \frac{1}{n!}\d_x^n \nu(t,\xb,dz).
\end{aligned} \right\} \label{e2}
\end{align}
In this case, \eqref{e10} and \eqref{eq:A.expand} become (respectively)
\begin{align}
\phi(t,x,\x)
    &=  \sum_{n=0}^\infty (x-\xb)^{n}\phi_{n}(t,\x) , &
\Ac
  &=  \sum_{n=0}^\infty (x-\xb)^{n}\phi_{n}(t,\Dc) ,
\end{align}
where, for all $n\ge 0$, the symbol $\phi_n(t,\xi)$ is given by \eqref{e8} with coefficients given by \eqref{e2}.
The choice of $\xb$ is somewhat arbitrary. However, a convenient choice that seems to work well in
most applications is to choose $\xb$ near $X_t$, the current level of $X$. Hereafter, to simplify notation, when discussing implementation of the Taylor-series expansion, we suppress the $\xb$-dependence: $a_n(t,\xb) \to a_n(t)$, $\gam_n(t,\xb) \to \gam_n(t)$
and $\nu_n(t,\xb,dz) \to \nu_n(t,dz)$.
\end{example}
\begin{example}[Two-point Taylor series expansion]\label{ex:two-pt}
Suppose $f$ is an analytic function with domain $\Rb$ and $\xb_1,\xb_2 \in \Rb$.  Then the \emph{two-point Taylor series} of $f$ is given by
\begin{align}
f(x)
    &=  \sum_{n=0}^\infty \( c_n(\xb_1,\xb_2)(x-\xb_1)+c_n(\xb_2,\xb_1)(x-\xb_2) \)
            (x-\xb_1)^n(x-\xb_2)^n , \label{eq:two-pt}
\end{align}
where
\begin{align}
c_0(\xb_1,\xb_2)
    &=  \frac{ f(\xb_2) }{\xb_2-\xb_1} , &
c_n(\xb_1,\xb_2)
    &=  \sum_{k=0}^n \frac{(k+n-1)!}{k!n!(n-k)!}
            \frac{(-1)^k k \d_{\xb_1}^{n-k}f(\xb_1) + (-1)^{n+1} n \d_{\xb_2}^{n-k}f(\xb_2)}{(\xb_1-\xb_2)^{k+n+1}} .
            \label{eq:c}
\end{align}
For the derivation of this result we refer the reader to \citet*{estes1966two,lopez2002two}.  Note
truncating the two-point Taylor series expansion \eqref{eq:two-pt} at $n=m$ results in an
expansion which of $f$ which is of order $\Oc(x^{2n+1}$).
\par
The advantage of using a two-point Taylor series is that, by considering the first $n$ derivatives
of a function $f$ at two points $\xb_1$ and $\xb_2$, one can achieve a more accurate approximation
of $f$ over a wider range of values than if one were to approximate $f$ using  $2n$ derivatives at
a single point (i.e., the usual Taylor series approximation).
\par
If we associate expansion \eqref{eq:two-pt} with an expansion of the form $f(x)=\sum_{n=0}^\infty f_n B_n(x)$ then $f_0 B_0(x)= c_n(\xb_1,\xb_2)(x-\xb_1)+c_n(\xb_2,\xb_1)(x-\xb_2)$, which is \emph{affine} in $x$.  Thus, the terms in the two-point Taylor series expansion would not be a suitable basis in \eqref{e10} since $B_0(x) \neq 1$. However, one can always introduce a constant $M$ and define a function
\begin{align}
F(x)
    &:= f(x) - M , &
    &\text{so that}&
f(x)
    &=  M + F(x) . \label{eq:M}
\end{align}
Then, one can express $f$ as
\begin{align}
f(x)
    &=  M + \sum_{n=1}^\infty \( C_{n-1}(\xb_1,\xb_2)(x-\xb_1)+C_{n-1}(\xb_2,\xb_1)(x-\xb_2) \)
            (x-\xb_1)^{n-1} (x-\xb_2)^{n-1} , \label{eq:two-pt.2}
\end{align}
where the $C_n$ are as given in \eqref{eq:c} with $f \to F$.  If we associate expansion \eqref{eq:two-pt.2} with an expansion of the form $f(x)=\sum_{n=0}^\infty f_n B_n(x)$, then we see that $f_0 B_0(x) = M$ and one can choose $B_0(x)=1$.  Thus, as written in \eqref{eq:two-pt.2}, the terms of the two-point Taylor series can be used as a suitable basis in \eqref{e10}.
\par
Consider the following case: suppose $a(t,x)$, $\gam(t,x)$ and $\nu(t,x,dz)$ are of the form
\begin{align}
a(t,x)
    &=f(x) A(t), &
\gam(t,x)
    &= f(x) \Gam(t), &
\nu(t,x,dz)
    &=  f(x) \Nc(t,dz),  \label{eq:proportional0}
\end{align}
so that $\phi(t,x,\xi)=f(x) \Phi(t,\xi)$ with
\begin{align}
\Phi(t,\xi)
    &= \Gam(t) ( i\x  - 1 ) + A(t) ( -\x ^2 - i \x  ) \\
                &\quad - \int_\Rb \Nc(t,dz) (e^z-1-z) i \x
                +  \int_\Rb \Nc(t,dz) ( e^{i \x  z} - 1 - i \x  z ) .
\end{align}
It is certainly plausible that the symbol of $\Ac$ would have such a form since, from a modeling
perspective, it makes sense that default intensity, volatility and jump-intensity would be
proportional. Indeed, the Jump-to-default CEV model (JDCEV) of \citet*{JDCEV, CarrMadan2010} has a
similar restriction on the form of the drift, volatility and killing coefficients.
\par
Now, under the dynamics of \eqref{eq:proportional0}, observe that $\phi(t,x,\xi)$ and $\Ac$ can be written as in \eqref{e10} and \eqref{eq:A.expand} respectively with $B_0=1$ and
\begin{align}
B_n(x)
    &=  \( C_{n-1}(\xb_1,\xb_2)(x-\xb_1)+C_{n-1}(\xb_2,\xb_1)(x-\xb_2) \)
            (x-\xb_1)^{n-1}(x-\xb_2)^{n-1} , \qquad n \geq 1 . \label{eq:B.two-pt}
\end{align}
As above $C_n$ (capital ``C'') are given by \eqref{eq:c} with $f \to F := f-M$ and
\begin{align}
\phi_0(t,\xi)
    &= M \Phi(t,\xi) , &
\phi_{n}(t,\xi)
    &=  \Phi(t,\xi) , \qquad n\geq 1.
\end{align}
As in example \ref{ex1}, the choice of $\xb_1$, $\xb_2$ and $M$ is somewhat arbitrary.  But, a choice that seems to work well is to set $\xb_1=X_t - \Delta$ and  $\xb_2=X_t + \Delta$ where $\Delta>0$ is a constant and $M = f(X_t)$.  It is also a good idea to check that, for a given choice of $\xb_1$ and $\xb_2$, the two-point Taylor series expansion provides a good approximation of $f$ in the region of interest.
\par
Note we assumed the form \eqref{eq:proportional0} only for sake of simplicity. Indeed, the
general case can be accommodated by suitably extending expansion \eqref{e10} to the more general
form
\begin{align}
\phi(t,x,\x)
    &=  \sum_{n=0}^\infty \sum_{i=1}^{3}B_{i,n}(x)\phi_{i,n}(t,\x) ,
\end{align}
where $\phi_{i,n}$ for $i=1,2,3$ are related to the diffusion, jump and default symbols
respectively. For brevity, however, we omit the details of the general case.
\end{example}
\begin{example}[{Non-local approximation in weighted $L^{2}$-spaces}]
\label{ex:L2}
Suppose $\{B_{n}\}_{n\ge 0}$ is a fixed orthonormal basis in some (possibly
weighted) space $L^{2}(\R,\mf(x)dx)$ and that $\phi(t,\cdot,\xi) \in L^{2}(\R,\mf(x)dx)$ for all
$(t,\xi)$. Then we can represent $\phi(t,x,\xi)$ in the form \eqref{e10} where now the $\{ \phi_n \}_{n \geq 0}$ are given by
\begin{align}
\phi_n(t,\x)&=\langle B_{n}(\cdot), \phi(t,\cdot,\x)\rangle_{\mf}, & n&\ge0.
\end{align}
A typical example would be to choose Hermite polynomials $H_n$ centered at $\xb$ as basis
functions, which (as normalized below) are orthonormal under a Gaussian weighting
\begin{align}
B_n(x)
    &=  H_n(x-\xb) , &
H_n(x)
    &:= \frac{1}{\sqrt{(2 n)!! \sqrt{\pi }}} \frac{\d_x^n \exp(-x^2)}{\exp(-x^2)} , &
n
    &\geq 0 . \label{eq:Bn.Hermite}
\end{align}
In this case, we have
\begin{align}
\phi_n(t,\x)
    &=  \< \phi(t,\cdot,\x) , B_n \>_{\mf}
    :=  \int_\Rb \phi(t,x,\x) B_n(x) \mf(x) dx , &
\mf(x)
    &:= \exp\big(-(x-\xb)^2\big) , \label{eq:phi.n.Hermite}
\end{align}
Once again, the choice of $\xb$ is arbitrary.  But, it is logical to choose $\xb$ near $X_t$, the
present level of the underlying $X$.
Note that, in the case of an $L^2$ orthonormal basis, differentiability of the coefficients $(a(t,\cdot),\gam(t,\cdot),\nu(t,\cdot,dz))$ is \emph{not} required.  This is a significant advantage over the Taylor and two-point Taylor basis functions considered in Examples \ref{ex1} and \ref{ex:two-pt}, which do require differentiability of the coefficients.
\end{example}
\par
Now, returning to Cauchy problem \eqref{eq:v.pide}, we suppose that $v=v(t,x)$ can be written as follows
\begin{align}
v
    &=  \sum_{n=0}^\infty v_n. \label{eq:v.expand}
\end{align}
Following \cite{RigaPagliaraniPascucci}, we insert expansions \eqref{eq:A.expand} and \eqref{eq:v.expand} into Cauchy problem \eqref{eq:v.pide} and find
\begin{align}
(\d_t + \Ac_0 ) v_0
    &=  0, &
v_0(T,x)
    &=  h(x), \label{eq:v.0.pide} \\
(\d_t + \Ac_0 ) v_n
    &= - \sum_{k=1}^{n} B_{k}(x) \Ac_k v_{n-k}, &
v_n(T,x)
    &=  0. \label{eq:v.n.pide}
\end{align}
We are now in a position to find the explicit expression for $\vh_n$, the Fourier transform
of $v_{n}$ in \eqref{eq:v.0.pide}-\eqref{eq:v.n.pide}.
\begin{theorem}
\label{thm:v}
Suppose $h \in L^1(\Rb,dx)$ and let $\hh$ denote its Fourier transform.  Suppose further that $v_n$ and its Fourier transform $\vh_n$ exist, and that both the left and right hand side of \eqref{eq:v.0.pide}-\eqref{eq:v.n.pide} belong to $L^1(\Rb,dx)$.
Then $\vh_n(t,\x)$ is given by
\begin{align}
\vh_0(t,\x)
    &=  \exp\left(\int_t^T  \phi_0(s, \x )ds \right) \hh(\x), \label{eq:v.hat.0} \\
\vh_n(t,\x)
    &=  \sum_{k=1}^n \int_t^T \exp\left(\int_t^s \phi_0(u,\x ) du\right)B_{k}(i \d_\x) \phi_k(s,\x) \vh_{n-k}(s,\x) ds, &
n
    &\geq 1. \label{eq:v.hat.n}
\end{align}
Note that the operator $B_{k}(i \d_\x)$ acts on everything to the right of it.
\end{theorem}
\begin{proof}
See Appendix \ref{sec:un}.
\end{proof}
\begin{remark}
To compute survival probabilities $v(t,x)=v(t,x;T)$ over the interval $[t,T]$, one assumes a payoff function $h(x) = 1$.  Note that the Fourier transform of a constant is simply a Dirac delta function:
$\hh(\x)=\delta(\x)$.  Thus, when computing survival probabilities, (possibly defaultable) bond prices and credit
spreads, no numerical integration is required.  Rather, one simply uses the following identity
\begin{align}
\int_\Rb \uh(\x) \d_\xi^n \del(\x) d\xi
    &=  (-1)^n \d_\xi^n \uh(\xi)|_{\xi=0}.
\end{align}
to compute inverse Fourier transforms.
\end{remark}
\begin{remark}
Assuming $\vh_n \in L^1(\Rb,dx)$, one recovers $v_n$ using
\begin{align}
v_n(t,x)
    &=  \int_\Rb d\xi\frac{1}{\sqrt{2\pi}}e^{i \x x} \vh_n(t, \xi). \label{eq:vn.FT}
\end{align}
As previously mentioned, to obtain the FK transition densities $p(t,x;T,y)$ one simply sets $h(x)
= \del_y(x)$.  In this case, $\hh(\xi)$ becomes $\psi_y(-\xi)$.
\end{remark}

\noindent When the coefficients $(a,\gam,\nu)$ are time-homogeneous, then the results of Theorem
\ref{thm:v} simplify considerably, as we show in the following corollary.
\begin{corollary}[Time-homogeneous case]
\label{thm:u} Suppose that $X$ has time-homogeneous dynamics with the local variance, default
intensity and L\'evy measure given by $a(x)$, $\gam(x)$ and $\nu(x,dz)$ respectively.  Then the
symbol $\phi_n(t,\x)=\phi_n(\x)$ is independent of $t$.  Define
\begin{equation}
\tau(t):= T-t.
\end{equation}
Then, for $n\leq 0$ we have
\begin{equation}
v_n(t,x)=u_n(\tau(t),x)
\end{equation}
where
\begin{align}
\uh_0(\tau,\x)
    &=  e^{\tau \phi_0(\x)} \hh(\x), \label{eq:u0.hat} \\
\uh_n(\tau,\x)
    &=  \sum_{k=1}^n \int_0^\tau e^{(\tau-s)\phi_0(\x)}B_{k}(i \d_\x) \phi_k(\x) \uh_{n-k}(s,\x)ds , &
n
    &\geq 1. \label{eq:uhatn}
\end{align}
\end{corollary}
\begin{proof}
The proof is an algebraic computation.  For brevity, we omit the details.
\end{proof}
\begin{example}
\label{ex:twopt.explicit} Consider the Taylor density expansion of Example \ref{ex1}. That is,
$B_n(x)=(x-\xb)^n$. Then, in the time-homogeneous case, we find that
$\uh_1(t,\x)$ and $\uh_2(t,\x)$ are given explicitly by
\begin{align}
\uh_1(t,\x)
    &=  e^{t \phi_0(\x  )} \( t \hh(\x  ) \xb \phi_1(\x) + i t \phi_1(\x  ) \hh'(\x)
            + \frac{1}{2} i t^2 \hh(\x) \phi_1(\x  ) \phi_0'(\x  )+i t \hh(\x) \phi_1'(\x) \), \label{eq:u1.hat} \\
\uh_2(t,\x)
    &=  e^{t \phi_0(\x  )} \bigg( \frac{1}{2} t^2 \hh(\x  ) \xb^2 \phi_1^2(\x  )
            +t \hh(\x  ) \xb^2 \phi_2(\x  )
            -i t^2 \xb \phi_1^2(\x  ) \hh'(\x  )
            -2 i t \xb \phi_2(\x  ) \hh'(\x  ) \\ & \qquad
            -\frac{1}{2} i t^3 \hh(\x  ) \xb \phi_1^2(\x  ) \phi_0'(\x  )
            -i t^2 \hh(\x  ) \xb \phi_2(\x  ) \phi_0'(\x  )
            -\frac{1}{2} t^3 \phi_1(\x  ){}^2 \hh'(\x  ) \phi_0'(\x  )
            -t^2 \phi_2(\x  ) \hh'(\x  ) \phi_0'(\x  )  \\ & \qquad
            -\frac{1}{8} t^4 \hh(\x  ) \phi_1^2(\x  ) (\phi_0'(\x  ))^2
            -\frac{1}{3} t^3 \hh(\x  ) \phi_2(\x  ) (\phi_0'(\x  ))^2
            -\frac{3}{2} i t^2 \hh(\x  ) \xb \phi_1(\x  ) \phi_1'(\x  ) \\ & \qquad
            -\frac{3}{2} t^2 \phi_1(\x  ) \hh'(\x  ) \phi_1'(\x  )
            -\frac{2}{3} t^3 \hh(\x  )\phi_1(\x  )\phi_0'(\x  ) \phi_1'(\x  )
            -\frac{1}{2} t^2 \hh(\x  ) (\phi_1'(\x  ))^2
            -2 i t \hh(\x  ) \xb \phi_2'(\x  ) \\ &\qquad
            -2 t \hh'(\x  ) \phi_2'(\x  )
            -t^2 \hh(\x  ) \phi_0'(\x  ) \phi_2'(\x  )
            -\frac{1}{2} t^2 \phi_1(\x  ){}^2 \hh''(\x  )
            -t \phi_2(\x  ) \hh''(\x  )
            -\frac{1}{6} t^3 \hh(\x  ) \phi_1^2(\x  ) \phi_0''(\x  ) \\ &\qquad
            -\frac{1}{2} t^2 \hh(\x  ) \phi_2(\x  ) \phi_0''(\x  )
            -\frac{1}{2} t^2 \hh(\x  ) \phi_1(\x  ) \phi_1''(\x  )-t \hh(\x  ) \phi_2''(\x  ) \bigg). \label{eq:u2.hat}
\end{align}
Higher order terms are quite long. However, they can be computed quickly and explicitly using the
Mathematica code provided in Appendix \ref{sec:mathematica}.  The code in the Appendix can be easily modified for use with other basis functions.
\end{example}

\begin{remark}
As in \citet{RigaPagliaraniPascucci}, when considering models with Gaussian-type jumps, i.e.,
models with a state-dependent L\'evy measure $\nu(t,x,dz)$ of the form \eqref{eqbound3} below, all
terms in the expansion for the transition density become explicit.  Likewise, for models with
Gaussian-type jumps, all terms in the expansion for the price of an option are also explicit,
assuming the payoff is integrable against Gaussian functions.
\end{remark}

\begin{remark}
Many common payoff functions (e.g. calls and puts) are not integrable: $h \notin L^1(\Rb,dx)$.  Such payoffs may sometimes be accommodated using \emph{generalized} Fourier transforms.  Assume
\begin{align}
\hh(\x)
    &:= \int_\Rb dx \frac{1}{\sqrt{2\pi}}e^{-i \x  x} h(x) < \infty, &
    &\text{for some $\x  = \x_r + i \x_i$ with $\x_r, \x_i \in \Rb$}.
\end{align}
Assume also that $\phi(t,x,\x_r + i \x_i)$ is analytic as a function of $\x_r$.
Then the formulas appearing in Theorem \ref{thm:v} and Corollary \ref{thm:u} are valid and
integration in \eqref{eq:vn.FT} is with respect to $\x_r$ (i.e., $d\x  \to
d\x_r$).  For example, the payoff of a European call option with payoff function
$h(x)=(e^x-e^k)^+$ has a generalized Fourier transform
\begin{align}
\hh(\x)
        &=      \int_\Rb dx \tfrac{1}{\sqrt{2\pi}}e^{-i \x  x} (e^x - e^k)^+
        =           \frac{-e^{k-i k \x }}{\sqrt{2 \pi } \(i \x  + \x ^2\)},  &
\x
    &=  \x _r + i \x _i, &
\x _r
    &\in \R, &
\x _i
    &\in (-\infty, -1) . \label{eq:hhat}
\end{align}
\end{remark}
In any practical scenario, one can only compute a finite number of terms in \eqref{eq:v.expand}.  Thus, we define $v^{(N)}$, the $N$th order approximation of $v$ by
\begin{align}
v^{(N)}(t,x)
    &=  \sum_{n=0}^N v_n(t,x)
    =       \int_\Rb d\x  \frac{1}{\sqrt{2\pi}}e^{i \x  x} \vh^{(n)}(t, \x ) , &
\vh^{(N)}(t,\x)
    &:= \sum_{n=0}^N \vh_n(t,\x), &
\end{align}
The function $u^{(N)}(t,x)$ (which we use for time-homogeneous cases) and the approximate FK transition density $p^{(N)}(t,x;T,y)$ are defined in an analogous fashion.

%
%

\section{Pointwise error bounds for Gaussian models}\label{sec:errors}
\label{errbou}

In this section we state some pointwise error estimates for $p^{(N)}(t,x;T,y)$, the $N$th order
approximation of the FK density of $(\p_{t}+\Ac)$ with $\Ac$ as in \eqref{eq:A}. Throughout this
Section, we assume Gaussian-type jumps with $(t,x)$-dependent mean, variance and jump intensities.
Furthermore, we work specifically with the Taylor series expansion of Example \ref{ex1}.  That is,
we use basis functions $B_n(x)=(x-\xb)^n$.
\begin{theorem}\label{t1}
Assume that
\begin{equation}
m\leq a(t,x)\leq M,\qquad 0\leq \gamma(t,x)\leq M, \qquad t\in[0,T],\ x\in\R,
\end{equation}
for some positive constants $m$ and $M$, and that
\begin{align}\label{eqbound3}
\nu(t,x,dz) &=\lam(t,x)\, \mathscr{N}_{\m(t,x),\delta^2(t,x)}(dz)
 :=\frac{\lam(t,x)}{\sqrt{2\pi}\delta(t,x)}e^{-\frac{(z-\m(t,x))^2}{2\delta^2(t,x)}}dz,
\end{align}
with
\begin{equation}
 m\leq  \delta^2(t,x)\leq M,\qquad 0\leq \lambda(t,x),|\mu(t,x)|\leq M, \qquad t\in[0,T],\ x\in\R.
\end{equation}
Moreover assume that $a,\g,\lambda,\delta,\mu$ and their $x$-derivatives are bounded and Lipschitz
continuous in $x$, and uniformly bounded with respect to $t\in[0,T]$. Let $\xb=y$ in \eqref{e2}.
Then, for $N\ge 1$, we have\footnote{Here $\left\|  \partial_x \nu  \right\|_{\infty}:=\max\{
\left\|
\partial_x\lambda  \right\|_{\infty},\left\|  \partial_x\delta  \right\|_{\infty},\left\|
\partial_x \mu  \right\|_{\infty} \}$, where $\|  \cdot  \|_{\infty}$ denotes the sup-norm on
$(0,T)\times\R$. Note that $\left\|\partial_{x}\nu\right\|_{\infty}=0$ if $\lambda,\delta,\mu$ are
constants.}
\begin{equation}\label{eqbound10}
 \left|p(t,x;T,y) - p^{(N)}(t,x;T,y)\right|\leq   g_{N}(T-t) \left(\bar{\G}(t,x;T,y)+\left\|  \partial_x \nu  \right\|_{\infty} \widetilde{\G}(t,x;T,y) \right),
\end{equation}
for any $x,y \in \Rb$ and $t< T$,  where
\begin{equation}
 g_{N}(s)=\Oc \left(s\right),\quad \textrm{as } s \to 0^+ .
\end{equation} Here, the function $\bar{\G}$ is the fundamental solution of the
constant coefficients jump-diffusion operator
\begin{align}\label{eq:integro_diff_constant}
  &\p_t u(t,x)+\frac{\bar{M}}{2}\p_{xx} +\bar{M}\int_{\R} \left( u(t,x+z)-u(t,x)\right) \mathscr{N}_{\bar{M},\bar{M}}(dz),
\end{align}
where $\bar{M}$ is a suitably large constant, and $\widetilde{\G}$ is defined as
\begin{equation}
\widetilde{\G}(t,x;T,y)=\sum_{k=0}^{\infty}
\frac{\bar{M}^{k/2}(T-t)^{k/2}}{\sqrt{k!}}
\mathscr{C}^{k+1}\bar{\G}(t,x;T,y),
\end{equation}
and where $\mathscr{C}$ is the convolution operator acting as
\begin{equation}
\mathscr{C}f(x)=\int_{\R}f(x+z)\mathscr{N}_{\bar{M},\bar{M}}(dz).
\end{equation}
\end{theorem}
\begin{proof}
An outline of the proof is provided in Appendix \ref{sec:errorproof}. For a detailed proof we refer to \cite{lpp-1.5}.
\end{proof}
\begin{remark}\label{remt1}
The functions $\mathscr{C}^k \bar{\G}$ take the
following form
\begin{align}\label{andconv}
\mathscr{C}^{k}\bar{\G}(t,x;T,y)&=e^{-\bar{M}(T-t)}\sum_{n=0}^{\infty}\frac{\big(\bar{M}
(T-t)\big)^n}{n!\sqrt{2\pi \bar{M} (T-t+n+k)}}\, \exp\left(-\frac{\left(x-y
+\bar{M}(n+k)\right)^2}  {2\bar{M}(T-t+n+k)}\right),\qquad  k\ge 0,
\end{align}
and therefore $\widetilde{\G}$ can be explicitly written as
\begin{equation}
\widetilde{\G}(t,x;T,y)=e^{-\bar{M}(T-t)}\sum_{n,k=0}^{\infty} \frac{\big(\bar{M}
(T-t)\big)^{n+\frac{k}{2}} }{n!\sqrt{k!}\sqrt{2\pi \bar{M} (T-t+n+k+1)}}  \,
\exp\left(-\frac{\left(x-y +\bar{M}(n+k+1)\right)^2}
  {2\bar{M}(T-t+n+k+1)}\right).
\end{equation}
\end{remark}
By Remark \ref{remt1}, it follows that, when $k=0$ and $x\neq y$, the asymptotic
behaviour as $t\to T$ of the sum in \eqref{andconv} depends only on the $n=1$ term.  Consequently, we have
$\bar{\G}(t,x;T,y)=\Oc (T-t)$ as $(T-t)$ tends to $0$. On the other hand, for $k\ge 1$,
$\mathscr{C}^{k}\bar{\G}(t,x;T,y)$, and thus also $\widetilde{\G}(t,x;T,y)$, tends to a positive
constant as $(T-t)$ goes to $0$. It is then clear by \eqref{eqbound10} that, with $x\neq y$ fixed,
the asymptotic behavior of the error, when $t$ tends to $T$, changes from $(T-t)$ to $(T-t)^2$
depending on whether the L\'evy measure is locally-dependent or not.

Theorem \ref{t1} extends the previous results in \cite{RigaPagliaraniPascucci} where only the
purely diffusive case (i.e $\lambda \equiv 0$) is considered. In that case an estimate analogous to
\eqref{eqbound10} holds with
 $$g_{N}(s)=\Oc \left(s^{\frac{N+1}{2}}\right),\quad \textrm{as } s \to 0^+.$$
Theorem \ref{t1} shows that for jump processes, increasing the order of the expansion for $N$
greater than one, theoretically does not give any gain in the rate of convergence of the
asymptotic expansion as $t \to T^-$; this is due to the fact that the expansion is based on a
local (Taylor) approximation while the PIDE contains a non-local part. This estimate is in accord
with the results in \cite{gobet-smart} where only the case of constant L\'evy measure is
considered. Thus Theorem \ref{t1} extends the latter results to state dependent Gaussian jumps
using a completely different technique. Extensive numerical tests showed that the first order
approximation gives extremely accurate results and the precision seems to be further improved by
considering higher order approximations.
\begin{corollary}\label{cor2}
Under the assumptions of Theorem \ref{t1}, we have the following estimate for the error on the
approximate prices:
\begin{equation}
 \left |v(t,x) - v^{(N)}(t,x)\right|\leq   g_{N}(T-t)  \int_{\mathbb{R}} |h(y)| \left(\bar{\G}(t,x;T,y)+\left\|  \partial_x \nu  \right\|_{\infty} \widetilde{\G}(t,x;T,y) \right) dy,
\end{equation}
for any $x\in \Rb$ and $t< T$.
\end{corollary}
Some possible extensions of these asymptotic error bounds to general L\'evy measures are possible,
though they are certainly not straightforward. Indeed, the proof of Theorem \ref{t1} is based on
some pointwise uniform estimates for the fundamental solution of the constant coefficient
operator, i.e. the transition density of a compound Poisson process with Gaussian jumps. When
considering other L\'evy measures these estimates would be difficult to carry out, especially
in the case of jumps with infinite activity, but they might be obtained in some suitable normed
functional space. This might lead to error bounds for short maturities, which are
expressed in terms of a suitable norm, as opposed to uniform pointwise bounds.
\begin{remark}
\label{rmk:practice}
Since, in general, it is hard to derive the truncation error bound, the reader may wonder how to determine the number of terms to  include in the asymptotic expansion.  Though we provide a general expression for the $n$-th term, realistically, only the fourth order term can be computed.  That said,
in practice, three terms provide an approximation which is accurate enough for most applications (i.e., the resulting approximation error is smaller than the bid-ask spread typically quoted on the market).  Since, $v^{(n)}$ only requires a single Fourier integration, there is no numerical advantage for choosing smaller $n$.  As such, for financial applications we suggest using $n=3$ or $n=4$.
\end{remark}

\section{Examples}
\label{sec:examples}

In this section, in order to illustrate the versatility of our asymptotic expansion, we apply our
approximation technique to a variety of different L\'evy-type models.  We consider both finite and
infinite activity L\'evy-type measures and models with and without a diffusion component.
We study not only option prices, but also implied volatilities.
In each setting, if the exact or approximate option price has been computed by a method other than our own, we compare this to the option
price
obtained by our approximation.  For cases where the exact
or approximate option price is not analytically available, we use Monte Carlo methods to
verify the accuracy of our method.
\par
Note that, some of the examples considered below do not satisfy the conditions listed in Section
\ref{sec:model}.   In particular, we will consider coefficients $(a,\gam,\nu)$ that are not bounded.  Nevertheless, the formal results of Section \ref{sec:formal} work well in the examples considered.


\subsection{CEV-like L\'evy-type processes}
\label{sec:CEV} We consider a L\'evy-type process of the form \eqref{eq:dX} with CEV-like
volatility and jump-intensity.  Specifically, the $\log$-price dynamics are given by
\begin{align}
a(x)
    &=  \frac{1}{2} \del^2 e^{2(\beta-1)x}, &
\nu(x,dz)
    &=  e^{2(\beta-1)x} \Nc(dz), &
\gam(x)
    &=  0, &
\del
    &\geq   0, &
\beta
    &\in [0,1], \label{eq:CEV.like}
\end{align}
where $\Nc(dx)$ is a L\'evy measure.  When $\Nc \equiv 0$, this model reduces to the CEV model of
\citet{CoxCEV}.  Note that, with $\beta \in [0,1)$, the volatility and jump-intensity increase as
$x \to -\infty$, which is consistent with the leverage effect (i.e., a decrease in the
value of the underlying is often accompanied by an increase in volatility/jump intensity).  This
characterization will yield a negative skew in the induced implied volatility surface. As the
class of models described by \eqref{eq:CEV.like} is of the form \eqref{eq:proportional0} with
$f(x)=e^{2(\beta-1)x}$, this class naturally lends itself to the two-point Taylor series
approximation of Example \ref{ex:two-pt}.  Thus, for certain numerical examples in this Section,
we use basis functions $B_n$ given by \eqref{eq:B.two-pt}.  In this case we choose expansion
points $\xb_1$ and $\xb_2$ in a symmetric interval around $X_0$ and in \eqref{eq:M} we choose $M =
f(X_0) =e^{2(\beta-1)X_0}$.  For other numerical examples, we use the (usual) one-point Taylor
series expansion $B_n(x)=(x-\xb)^n$.  In this cases, we choose $\xb = X_0$.
\par
We will consider two different characterizations of $\Nc(dz)$:
\begin{align}
\text{Gaussian:}&&
\Nc(dz)
    &=  \lam \frac{1}{\sqrt{2\pi \eta^2}}\exp\( \frac{-(z-m)^2}{2 \eta^2} \) dz,
            \label{eq:eta.gaussian} \\
\text{Variance-Gamma:}&&
\Nc(dz)
    &=  \( \frac{e^{-\lam_{-} |z|}}{\kappa|z|} \Ib_{\{z<0\}} + \frac{e^{-\lam_{+} z}}{\kappa z} \Ib_{\{z>0\}} \) dz,
            \label{eq:eta.VG} \\
&&
\lam_{\pm}
    &=  \( \sqrt{\frac{\theta^2 \kappa^2}{4} + \frac{\rho^2 \kappa}{2}} \pm \frac{\theta \kappa}{2}\)^{-1}
\end{align}
Note that the Gaussian measure is an example of a finite-activity L\'evy measure (i.e., $\Nc(\Rb)<\infty$), whereas the Variance-Gamma measure, due to \citet{madan1998variance}, is an infinite-activity L\'evy measure (i.e., $\Nc(\Rb) = \infty$).  As far as the authors of this paper are aware, there is no closed-form expression for option prices (or the transition density) in the setting of \eqref{eq:CEV.like}, regardless of the choice of $\Nc(dz)$.  As such, we will compare our pricing approximation to prices of options computed via standard Monte Carlo methods.
\begin{remark}
\footnote{We would like to thank an anonymous referee for bringing the issue of boundary conditions to our attention.}
\label{rem:boundarycond}
Note, the CEV model typically includes an absorbing boundary condition at $S=0$.  A more rigorous way to deal with degenerate dynamics, as in the CEV model, would be to approximate the solution of the Cauchy problem related to the process $S_t$ (as apposed to $X_t = \log S_t$).  One would then equip the Cauchy problem with suitable Dirichlet conditions on the boundary $s=0$, and work directly in the variable $s \in \R_{+}$ as opposed to the log-price on $x \in \R$.  Indeed, this is the approach followed by \citet*{hagan-woodward} who approximate the true density $p$ by a Gaussian density $p_{0}$ through a heat kernel expansion: note that the supports of $p$ and $p_{0}$ are $\R_{+}$ and $\R$ respectively. In order to take into account of the boundary behavior of the true density $p$, an improved approximation could be achieved by using the Green function of the heat operator for $\R_{+}$ instead of the Gaussian kernel: this will be object of further research in a forthcoming paper.
\par
We would also like to remark explicitly that our methodology is very general and works with different choices for the leading operator of the expansion, such as the constant-coefficient PIDEs we consider in the case of jumps.  Nevertheless, in the present paper, when purely diffusive models are considered, {\it we always take the heat operator as the leading term of our expansion}. The main reasons are that (i) the heat kernel is convenient for its computational simplicity and (ii) the heat kernel allows for the possibility of passing directly from a Black-Scholes-type price expansion to an implied vol expansion.
%
\end{remark}
\subsubsection{Gaussian L\'evy Measure}
\label{sec:Merton.CEV}
In our first numerical experiment, we consider the case of Gaussian jumps.
That is, $\Nc(dz)$ is given by \eqref{eq:eta.gaussian}.  We fix the following parameters
\begin{align}
\del
    &= 0.20, &
\beta
    &= 0.25, &
\lam
    &= 0.3, &
m
    &= -0.1, &
\eta
    &=  0.4, &
S_0 = e^x
    &=  1. \label{eq:parameters}
\end{align}
Using Corollary \ref{thm:u}, we compute the approximate prices $u^{(0)}(t,x;K)$ and
$u^{(3)}(t,x;K)$ of a series of European puts over a range of strikes $K$ and with times to
maturity $t=\{0.25,1.00,3.00,5.00\}$ (we add the parameter $K$ to the arguments of $u^{(n)}$ to
emphasize the dependence of $u^{(n)}$ on the strike price $K$). To compute $u^{(i)}(t,x;K)$, $i=\{0,3\}$ we use the we the usual one-point Taylor series expansion (Example \ref{ex1}).
We also compute the price $u(t,x;K)$ of each put by Monte
Carlo simulation. For the Monte Carlo simulation, we use a standard Euler scheme with a time-step
of $10^{-3}$ years, and we simulate $10^6$ sample paths. We denote by $u^{(MC)}(t,x;K)$ the price
of a put obtained by Monte Carlo simulation.
As prices are often quoted in implied volatilities, we convert prices to implied volatilities by inverting the Black-Scholes formula numerically.  That is, for a given put price $u(t,x;K)$, we find $\sig(t,K)$ such that
\begin{align}
u(t,x;K)
    &=  u^{\text{\rm BS}}(t,x;K,\sig(t,K)), \label{eq:impvol.solve}
\end{align}
where $u^{\text{\rm BS}}(t,x;K,\sig)$ is the Black-Scholes price of the put as computed assuming a
Black-Scholes volatility of $\sig$.  For convenience, we introduce the notation
\begin{align}
\text{IV}[u(t,x;K)]:= \sig(t,K)
\end{align}
to indicate the implied volatility induced by option price $u(t,x;K)$.
The results of our numerical experiments are plotted in Figure
\ref{fig:IV-Gauss-1}.  We observe that $\text{IV}[u^{(3)}(t,x;K)]$ agrees almost exactly with $\text{IV}[u^{(MC)}(t,x;K)]$.   The computed prices $u^{(3)}(t,x;K)$ and their induced implied volatilities $\text{IV}[u^{(3)}(t,x;K)]$, as well as 95\% confidence intervals resulting from the Monte Carlo simulations can be found in Table \ref{tab:IV-Gauss-1}.

\subsubsection*{Comparing one-point Taylor and Hermite expansions}
As choosing different basis functions results in distinct option-pricing approximations, one might wonder: which choice of basis functions provides the most accurate approximation of option prices and implied volatilities?  We investigate this question in Figure \ref{fig:OnePt-Hermite}.  In the left column, using the parameters in \eqref{eq:parameters}, we plot $\text{IV}[u^{(n)}(t,x;K)]$, $t=0.5$, $n=\{0,1,2,3,4\}$ where $u^{(n)}(t,x;K)$ is computed using both the one-point Taylor series basis functions (Example \ref{ex1}) and the Hermite polynomial basis functions (Example \ref{ex:L2}).  We also plot $\text{IV}[u^{(MC)}(t,x;K)]$, the implied volatility obtained by Monte Carlo simulation.  For comparison, in the right column, we plot the function $f$ as well as $f_{\text{Taylor}}^{(n)}$ and $f_{\text{Hermite}}^{(n)}$ where
\begin{align}
f(x)
    &=e^{2(\beta-1)x} , &
f_{\text{Taylor}}^{(n)}(x)
    &:= \sum_{m=0}^n \frac{1}{m!} \d^m f(\xb) (x-\xb)^m , &
f_{\text{Hermite}}^{(n)}(x)
    &:= \sum_{m=0}^n \frac{1}{m!} \< H_m , f \> H_m(x) . \label{eq:fs}
\end{align}
From Figure \ref{fig:OnePt-Hermite}, we observe that, for every $n$, the Taylor series expansion $f_{\text{Taylor}}^{(n)}$ provides a better approximation of the function $f$ (at least locally) than does the Hermite polynomial expansion $f_{\text{Hermite}}^{(n)}$.  In turn, the implied volatilities resulting from the Taylor series basis functions $\text{IV}[u^{(n)}(t,x;K)]$ more accurately approximate $\text{IV}[u^{(MC)}(t,x;K)]$ than do the implied volatilities resulting from the Hermite basis functions.  The implied volatilities resulting from the two-point Taylor series price approximation (not shown in the Figure for clarity), are nearly indistinguishable implied volatilities induced by the (usual) one-point Taylor series price approximation.

\subsubsection*{Computational speed, accuracy and robustness}
In order for a method of obtaining approximate option prices to be useful to practitioners, the method must be fast, accurate and work over a wide range of model parameters.  In order to test the speed, accuracy and robustness of our method, we select model parameters at random from uniform distributions within the following ranges
\begin{align}
\del
    &\in [0.0,0.6] , &
\beta
    &\in [0.0,1.0] , &
\lambda
    &\in [0.0,1.0] , &
m
    &\in [-1.0,0.0] , &
\eta
    &\in [0.0,1.0] . &
\end{align}
Using the obtained parameters, we then compute approximate option prices $u^{(3)}$ and record computation times over a fixed range of strikes using our third order one-point Taylor expansion (Example \ref{ex1}).  As the exact price of a call option is not available, we also compute option prices by Monte Carlo simulation.  The results are displayed in Tables \ref{tab:IV-Gauss-2} and \ref{tab:IV-Gauss-3}.  The tables show that our third order price approximation $u^{(3)}$ consistently falls within the 95\% confidence interval obtained from the Monte Carlo simulation.  Moreover, using a 2.4 GHz laptop computer, an approximate call price $u^{(3)}$ can be computed in only $\approx 0.05$ seconds.  This is only four to five times larger than the amount of time it takes to compute a similar option price using standard Fourier methods in an exponential L\'evy setting.

\subsubsection{Variance Gamma L\'evy Measure}
\label{sec:VG.CEV}
In our second numerical experiment, we consider the case of Variance Gamma jumps.  That is, $\Nc(dz)$ given by \eqref{eq:eta.VG}.  We fix the following parameters:
\begin{align}
\del
    &= 0.0, &
\beta
    &= 0.25, &
\theta
    &= -0.3, &
\rho
    &= 0.3, &
\kappa
    &=  0.15, &
S_0 = e^x
    &=  1.
\end{align}
Note that, by letting $\del=0$, we have set the diffusion component of $X$ to zero: $a(x)=0$.
Thus, $X$ is a pure-jump L\'evy-type process.  Using Corollary \ref{thm:u}, we compute the
approximate prices $u^{(0)}(t,x;K)$ and $u^{(2)}(t,x;K)$ of a series of European puts over a range
of strikes and with maturities $t \in \{0.5, 1.0\}$.  To compute $u^{(i)}$, $i \in \{0,2 \}$, we use the two-point Taylor series expansion (Example \ref{ex:two-pt}).  We also compute the put prices by Monte Carlo simulation.  For the Monte Carlo simulation, we use a time-step of $ 10^{-3}$ years and we simulate $10^6$ sample paths.  At each time-step, we update $X$ using the following algorithm
\begin{align}
X_{t+\Delta t}
    &=  X_t + b(X_t) \Delta t + \gam^+(X_t) - \gam^-(X_t), &
I(x)
    &=  e^{2(\beta-1)x}, \\
b(x)
    &=  - \frac{I(x)}{\kappa} \( \log \( \frac{\lam_-}{1+\lam_-}\) + \log\( \frac{\lam_+}{\lam_+-1} \) \), &
\gam^\pm(x)
    &\sim \Gam( I(x) \cdot \Delta t/ \kappa, 1/\lam_\pm),
\end{align}
where $\Gam(a,b)$ is a Gamma-distributed random variable with shape parameter $a$ and scale parameter $b$.  Note that this is equivalent to considering a VG-type process with state-dependent parameters
\begin{align}
\kappa'(x)
    &:= \kappa/I(x), &
\theta'(x)
    &:= \theta I(x), &
\rho'(x)
    &:= \rho \sqrt{I(x)}.
\end{align}
These state-dependent parameters result in state-\emph{independent} $\lam_\pm$ (i.e., $\lam_\pm$ remain constant).
Once again, since implied volatilities rather than prices are the quantity of primary interest, we convert prices to implied volatilities by inverting the Black-Scholes formula numerically.  The results are plotted in Figure \ref{fig:IV-VG-MonteCarlo}.  We observe that $\text{IV}[u^{(2)}(t,x;K)]$ agrees almost exactly with $\text{IV}[u^{(MC)}(t,x;K)]$.  Values for $u^{(2)}(t,x;K)$, the associated implied volatilities $\text{IV}[u^{(2)}(t,x;K)]$ and the 95\% confidence intervals resulting from the Monte Carlo simulation can be found in table \ref{tab:IV-VG-MonteCarlo1}.

\section{Conclusion}
\label{sec:conclusion}

In this paper, we consider an asset whose risk-neutral dynamics are described by an exponential
L\'evy-type martingale subject to default. This class includes nearly all non-negative Markov
processes.  In this very general setting, we provide a family of approximations -- one for each
choice of the basis functions (i.e. Taylor, two-point Taylor, $L^2$ basis, etc.) -- for (i) the
transition density of the underlying (ii) European-style option prices and their sensitivities and
(iii) defaultable bond prices and their credit spreads.  For the transition densities, {and thus
for option and bond prices as well}, we establish the accuracy of our asymptotic expansion.

\subsection*{Thanks}
The authors would like to extend their thanks to an anonymous referee, whose comments and suggestions helped to improve this paper.

%
%

\appendix


\section{Proof of Theorem \ref{thm:v}}
\label{sec:un}

By hypothesis $v_n \in L^1(\Rb,dx)$, and thus, by standard Fourier transform properties we the
following relation holds:
\begin{equation}\label{Fourierprop}
 \F(\Ac_k v_{n}(t,\cdot))(\xi)=\phi_k(t,\x) \vh_{n}(t,\x),\quad n,k\geq 0.
\end{equation}
We now Fourier transform equation \eqref{eq:v.n.pide}.  At the left-hand side we have
\begin{equation}
 \F(\(\d_t + \Ac_0 \) v_n(t,\cdot))(\xi)=\left(\d_t + \phi_0(t,\x)\right) \vh_n(t,\x).
\end{equation}
Next, for the right-hand side of \eqref{eq:v.n.pide} we get
\begin{align}
 -\sum_{k=1}^n \int_\Rb dx \left(\frac{e^{-i \x  x}}{\sqrt{2\pi}} B_{k}(x)\right) \Ac_k v_{n-k}(t,x)
    &=  - \sum_{k=1}^n \int_\Rb dx \left(B_{k}(i\p_{\x})\frac{e^{-i \x  x}}{\sqrt{2\pi}}\right) \Ac_k v_{n-k}(t,x) \\
    &=  - \sum_{k=1}^n B_{k}(i\p_{\x})\F(\Ac_k  v_{n-k}(t,\cdot))(\xi)
    \intertext{(by \eqref{Fourierprop})}
        &=  -\sum_{k=1}^n B_{k}(i\p_{\x}) \left(\phi_k(t,\x) \vh_{n-k}(t,\x)\right).
\end{align}
Thus, we have the following ODEs (in $t$) for $\vh_n(t,\x)$
\begin{align}
 \left( \d_t + \phi_0(t,\x)\right) \vh_0(t,\x)
    &=  0, &
\vh_0(T,\x )
    &=  \hh(\x), \label{eq:v.hat.0.ode} \\
 \left(\d_t + \phi_0(t,\x)\right) \vh_n(t,\x)
    &=  - \sum_{k=1}^n B_{k}(i\p_{\x}) \left(\phi_k(t,\x) \vh_{n-k}(t,\x)\right)
    &\vh_n(T,\x ) &=  0. \label{eq:v.hat.n.ode}
\end{align}
One can easily verify (e.g., by substitution) that the solutions of \eqref{eq:v.hat.0.ode} and \eqref{eq:v.hat.n.ode} are given by \eqref{eq:v.hat.0} and \eqref{eq:v.hat.n} respectively.

\endproof


\section{Mathematica code}
\label{sec:mathematica}

The following Mathematica code can be used to generate the $\uh_n(t,\x)$ automatically for Taylor
series basis functions: $B_n(x)=(x-x_0)^n$.  We have
\begin{align}
\texttt{B}[\texttt{n$\_$},\texttt{x$\_$},\texttt{x0$\_$}]
    &=(\texttt{x}-\texttt{x0}){}^{\wedge}\texttt{n};\\
\texttt{Bop}[\texttt{n$\_$},\xi \_,\texttt{x0$\_$},\texttt{ff$\_$}]
    &\texttt{:=}\texttt{Module}\Big[\{\texttt{mat},\dim ,\texttt{x}\},\\
    &\texttt{mat}
    = \texttt{CoefficientList}[\texttt{B}[n,x,\texttt{x0}],\texttt{x}];\\
    &\dim
    = \texttt{Dimensions}[\texttt{mat}];\\
    &\sum _{\texttt{m}=1}^{\dim [[1]]} \texttt{mat}[[\texttt{m}]](\texttt{i}){}^{\wedge}(\texttt{m}-1)\texttt{D}[\texttt{ff},\{\xi ,\texttt{m}-1\}]\\
\Big]; \\ \texttt{u}[\texttt{n$\_$},\texttt{t$\_$},\xi \_,\texttt{x0$\_$},\texttt{k$\_$}]
    &\texttt{:=}\texttt{Exp}[\texttt{t} \texttt{$\phi$}[0,\xi ,\texttt{x0}]] \sum _{\texttt{m}=1}^\texttt{n} \\ &\qquad
        \int_0^\texttt{t}
        \texttt{Exp}[-\texttt{s} \texttt{$\phi $}[0,\xi ,\texttt{x0}]](\texttt{Bop}[\texttt{m},\xi ,\texttt{x0},\texttt{$\phi$}[\texttt{m},\xi ,\texttt{x0}]\texttt{u}[\texttt{n}-\texttt{m},\texttt{s},\xi ,\texttt{x0},\texttt{k}]])\texttt{ds};\\
\texttt{u}[0,\texttt{t$\_$},\xi \_\texttt{x0$\_$},\texttt{k$\_$}]
    &=\texttt{Exp}[\texttt{t} \texttt{$\phi $}[0,\xi ,\texttt{x0}]]\texttt{h}[\xi ,\texttt{k}] ;
\end{align}
The function $\uh_n(t,\x)$ is now computed explicitly by typing
$\texttt{u}[\texttt{n$\_$},\texttt{t$\_$},\xi \_,\texttt{x0$\_$},\texttt{k$\_$}]$ and pressing
Shift+Enter.  Note that the function $\uh_n(t,\x)$ can depend on a parameter $k$ (e.g.,
$\log$-strike) through the Fourier transform of the payoff function $\hh(\x,k)$. To compute
$\uh_n(t,\x)$ using other basis functions, one simply has to replace the first line in the code.
For example, for Hermite polynomial basis functions, one re-writes the top line as
\begin{align}
\texttt{B}[\texttt{n$\_$},\texttt{x$\_$},\texttt{x0$\_$}]
    &=\frac{1}{\sqrt{(\texttt{2n})\texttt{!!}\sqrt{\pi }}}\texttt{HermiteH}[\texttt{n},\texttt{x}-\texttt{x0}];
\end{align}
where $\texttt{HermiteH}[\texttt{n},\texttt{x}]$ is the Mathematica command for the $n$-th Hermite
polynomial $H_n(x)$ (note that Mathematica does not normalize the Hermite polynomials as we do in
equation \eqref{eq:Bn.Hermite}).


\section{Sketch of the Proof of Theorem \ref{t1}}\label{sec:errorproof}

\proof

For sake of simplicity we only provide a sketch of the proof for the case $\delta(t,x)\equiv\delta$, $\mu(t,x)\equiv\mu$ and $N=1$. For the complete and detailed proof we refer to the companion paper \cite{lpp-1.5}.

The main idea of the proof is to use our expansion as a \emph{parametrix}.  That is, our expansion will be the starting point of the classical iterative method introduced by \cite{Levi} to construct the fundamental solution $p(t,x;T,y)$.
%
Specifically, as in \cite{RigaPagliaraniPascucci}, we take as a parametrix our $N$-th order
approximation $p^{(N)}(t,x;T,y)$ with basis functions $B_n=(x-\xb)^n$ and with $\overline{x}=y$. 
By analogy with the classical approach (see, for instance, \cite{Friedman} and
\cite{DiFrancescoPascucci2}, \cite{Pascucci2011} for the purely diffusive case, or
\cite{GarroniMenaldi} for the
integro-differential case), we have 
\begin{equation}\label{eqbound2}
 p(t,x;T,y)=p^{(1)}(t,x;T,y)+\int_{t}^{T}\int_{\R}p^{(0)}(t,x;s,\x)\Phi(s,\x;T,y)d\x ds ,
\end{equation}
where $\Phi$ is 
determined by imposing the condition 
  $$0=L p(t,x;T,y)=L p^{(1)}(t,x;T,y)+\int_{t}^{T}\int_{\R}L p^{(0)}(t,x;s,\x)\Phi(s,\x;T,y)d\x ds
  -\Phi(t,x;T,y).$$
Equivalently, we have
  $$\Phi(t,x;T,y)=L p^{(1)}(t,x;T,y)+\int_{t}^{T}\int_{\R}L p^{(0)}(t,x;s,\x)\Phi(s,\x;T,y)d\x ds , $$
and therefore by iteration
  \begin{equation}\label{eqbound1}
  \Phi(t,x;T,y)=\sum_{n=0}^{\infty}Z_{n}(t,x;T,y),
  \end{equation}
where
  \begin{align}\label{eqbound6}
  Z_{0}(t,x;T,y) & :=L p^{(1)}(t,x;T,y), \\ \label{eqbound9}
  Z_{n+1}(t,x;T,y)& :=\int_{t}^{T}\int_{\R}L p^{(0)}(t,x;s,\x)Z_{n}(s,\x;T,y)d\x ds.
  \end{align}
The proof of Theorem \ref{t1}, then, is based on some pointwise bounds
for each term $Z_{n}$ in \eqref{eqbound1}. 
These bounds, summarized in the next two propositions, can be combined with formula \eqref{eqbound2} to obtain the estimate for $\left|p(t,x;T,y)- p^{(1)}(t,x;T,y)\right|$.

\begin{proposition}\label{lemestim1}
There exists a positive constant $C$, only dependent on
$\t,m$ and $M$, such that
  \begin{equation}\label{e20}
  p^{(0)}(t,x;T,y)\leq C\,  \bar{\G}(t,x;T,y),\quad 0\leq t < T \leq \tau.
  \end{equation}
  for any $x,y\in\R$ and $t,T\in\R$ with $0\leq t<T\le \t$.
\end{proposition}

\begin{proposition}\label{properr3}
There exists a positive constant $C$, only dependent on
$\t,m,M$ and $(\|\lambda_{i}\|_{\infty}$,$\|\gamma_{i}\|_{\infty},\|a_{i}\|_{\infty})_{i=1,2}$, such that
\begin{equation}\label{eqbound7}
\left|Z_n(t,x;T,y) \right| \leq \frac{C^{n+1}
(T-t)^{\frac{n}{2}}}{\sqrt{n!}}\left(1+\|\lambda_{1}\|_{\infty}\mathscr{C}^{n+1}\right)\,
\bar{\G}(t,x;T,y),
\end{equation}
for any $n\ge 0$, $x,y\in\R$ and $t,T\in\R$ with $0\leq t<T\le \t$.
\end{proposition}
The proofs of Proposition \ref{lemestim1} and Proposition \ref{properr3} are rather technical and are based on several global pointwise estimates for the fundamental solution of a constant coefficient integro-differential  operator of the form \eqref{eq:integro_diff_constant}, along with the semigroup property
\begin{equation}\label{eqbound20}
 \int_{\R} \mathscr{C}^{k}\bar{\G}(t,x;s,\x)\,
 \mathscr{C}^{N}\bar{\G}(s,\x;T,y)\, d\x
 =\mathscr{C}^{k+N}\bar{\G}(t,x;T,y),\qquad  k,N\geq 0.
\end{equation}
We refer to \cite{lpp-1.5} for detailed proofs.

Now, by equations \eqref{eqbound2}, \eqref{eqbound1} and Proposition \ref{properr3} we have
\begin{align*}
 &|p(t,x;T,y)-p^{(1)}(t,x;T,y)|&\\
 &\leq \sum_{n=0}^{\infty}
 \frac{C^{n+1}}{\sqrt{n!}} \int_t^T (T-s)^{\frac{n}{2}} \int_{\R}
 p^{(0)}(t,x;s,\x)\  \left(1+\|\lam_{1}\|_{\infty} \mathscr{C}^{n+1}\right)  \bar{\G}(s,\x;T,y)  d\x ds&\\
%
\intertext{(by Proposition \ref{lemestim1}
)} &\leq \sum_{n=0}^{\infty}
 \frac{C^{n+1}}{\sqrt{n!}} \int_t^T (T-s)^{\frac{n}{2}} \int_{\R}
 \bar{\G}(t,x;s,\x)\  \left(1+\|\lam_{1}\|_{\infty} \mathscr{C}^{n+1}\right)  \bar{\G}(s,\x;T,y)  d\x ds&\\
\intertext{(by the semi-group property \eqref{eqbound20})}
  & =2(T-t)
 \left( \sum_{n=0}^{\infty} \frac{C^{n+1}
 (T-t)^{\frac{n}{2}}}{\sqrt{n!}} \left(1+\|\lam_{1}\|_{\infty} \mathscr{C}^{n+1}\right)\bar{\G}(t,x;T,y) \right),
\end{align*}
for any $x,y\in\R$ and $t,T\in\R$ with $0\leq t<T \le \t$.  Since
$$
\sum_{n=0}^{\infty} \frac{C^{n+1}
 (T-t)^{\frac{n}{2}}}{\sqrt{n!}}\,  \mathscr{C}^{n+1}\bar{\G}(t,x;T,y) ,
$$ can be easily checked to be convergent, this concludes the proof.
\endproof

\bibliographystyle{chicago}
\bibliography{LPP-bib}

\begin{thebibliography}{}

\bibitem[\protect\citeauthoryear{Andersen and Andreasen}{Andersen and
  Andreasen}{2000}]{andersen2000jump}
Andersen, L. and J.~Andreasen (2000).
\newblock Jump-diffusion processes: Volatility smile fitting and numerical
  methods for option pricing.
\newblock {\em Review of Derivatives Research\/}~{\em 4\/}(3), 231--262.

\bibitem[\protect\citeauthoryear{Benhamou, Gobet, and Miri}{Benhamou
  et~al.}{2009}]{gobet-smart}
Benhamou, E., E.~Gobet, and M.~Miri (2009).
\newblock Smart expansion and fast calibration for jump diffusions.
\newblock {\em Finance and Stochastics\/}~{\em 13\/}(4), 563--589.

\bibitem[\protect\citeauthoryear{Bielecki and Rutkowski}{Bielecki and
  Rutkowski}{2001}]{bielecki2001credit}
Bielecki, T. and M.~Rutkowski (2001).
\newblock {\em Credit Risk: Modelling, Valuation and Hedging}.
\newblock Springer.

\bibitem[\protect\citeauthoryear{Boyarchenko and Levendorskii}{Boyarchenko and
  Levendorskii}{2002}]{levendorskiibook}
Boyarchenko, S. and S.~Levendorskii (2002).
\newblock {\em Non-Gaussian Merton-Black-Scholes Theory}.
\newblock World Scientific.

\bibitem[\protect\citeauthoryear{Capponi, Pagliarani, and Vargiolu}{Capponi
  et~al.}{2013}]{capponi}
Capponi, A., S.~Pagliarani, and T.~Vargiolu (2013).
\newblock Pricing vulnerable claims in a {L}{\'e}vy driven model.
\newblock {\em preprint SSRN\/}.

\bibitem[\protect\citeauthoryear{Carr, Geman, Madan, and Yor}{Carr
  et~al.}{2002}]{CGMY}
Carr, P., H.~Geman, D.~Madan, and M.~Yor (2002).
\newblock The fine structure of asset returns: An empirical investigation.
\newblock {\em The Journal of Business\/}~{\em 75\/}(2), 305--333.

\bibitem[\protect\citeauthoryear{Carr and Linetsky}{Carr and
  Linetsky}{2006}]{JDCEV}
Carr, P. and V.~Linetsky (2006).
\newblock A jump to default extended {CEV} model: An application of {B}essel
  processes.
\newblock {\em Finance and Stochastics\/}~{\em 10\/}(3), 303--330.

\bibitem[\protect\citeauthoryear{Carr and Madan}{Carr and
  Madan}{2010}]{CarrMadan2010}
Carr, P. and D.~B. Madan (2010).
\newblock Local volatility enhanced by a jump to default.
\newblock {\em SIAM J. Financial Math.\/}~{\em 1}, 2--15.

\bibitem[\protect\citeauthoryear{Christof\-fersen, Jacobs, and
  Ornthanalai}{Christof\-fersen et~al.}{2009}]{christoffersen2009}
Christof\-fersen, P., K.~Jacobs, and Ornthanalai (2009).
\newblock {\em Exploring Time-Varying Jump Intensities: Evidence from S\&P500
  Returns and Options}.
\newblock CIRANO.

\bibitem[\protect\citeauthoryear{Cox}{Cox}{1975}]{CoxCEV}
Cox, J. (1975).
\newblock Notes on option pricing {I}: Constant elasticity of diffusions.
\newblock {\em Unpublished draft, Stanford University\/}.
\newblock A revised version of the paper was published by the Journal of
  Portfolio Management in 1996.

\bibitem[\protect\citeauthoryear{Di~Francesco and Pascucci}{Di~Francesco and
  Pascucci}{2005}]{DiFrancescoPascucci2}
Di~Francesco, M. and A.~Pascucci (2005).
\newblock On a class of degenerate parabolic equations of {K}olmogorov type.
\newblock {\em AMRX Appl. Math. Res. Express\/}~{\em 3}, 77--116.

\bibitem[\protect\citeauthoryear{Dupire}{Dupire}{1994}]{dupire1994pricing}
Dupire, B. (1994).
\newblock {Pricing with a smile}.
\newblock {\em Risk\/}~{\em 7\/}(1), 18--20.

\bibitem[\protect\citeauthoryear{Eraker}{Eraker}{2004}]{eraker}
Eraker, B. (2004).
\newblock Do stock prices and volatility jump? {R}econciling evidence from spot
  and option prices.
\newblock {\em The Journal of Finance\/}~{\em 59\/}(3), 1367--1404.

\bibitem[\protect\citeauthoryear{Estes and Lancaster}{Estes and
  Lancaster}{1972}]{estes1966two}
Estes, R.~H. and E.~R. Lancaster (1972).
\newblock Some generalized power series inversions.
\newblock {\em SIAM J. Numer. Anal.\/}~{\em 9}, 241--247.

\bibitem[\protect\citeauthoryear{Friedman}{Friedman}{1964}]{Friedman}
Friedman, A. (1964).
\newblock {\em Partial differential equations of parabolic type}.
\newblock Englewood Cliffs, N.J.: Prentice-Hall Inc.

\bibitem[\protect\citeauthoryear{Garroni and Menaldi}{Garroni and
  Menaldi}{1992}]{GarroniMenaldi}
Garroni, M.~G. and J.-L. Menaldi (1992).
\newblock {\em Green functions for second order parabolic integro-differential
  problems}, Volume 275 of {\em Pitman Research Notes in Mathematics Series}.
\newblock Harlow: Longman Scientific \& Technical.

\bibitem[\protect\citeauthoryear{Hagan and Woodward}{Hagan and
  Woodward}{1999}]{hagan-woodward}
Hagan, P. and D.~Woodward (1999).
\newblock Equivalent {B}lack volatilities.
\newblock {\em Applied Mathematical Finance\/}~{\em 6\/}(3), 147--157.

\bibitem[\protect\citeauthoryear{Hoh}{Hoh}{1998}]{hoh1998pseudo}
Hoh, W. (1998).
\newblock Pseudo differential operators generating {M}arkov processes.
\newblock {\em Habilitations-schrift, Universit{\"a}t Bielefeld\/}.

\bibitem[\protect\citeauthoryear{{Jacquier} and {Lorig}}{{Jacquier} and
  {Lorig}}{2013}]{lorigCEVLevy}
{Jacquier}, A. and M.~{Lorig} (2013).
\newblock The smile of certain {L}{\'e}vy-type models.
\newblock {\em ArXiv preprint arXiv:1207.1630\/}.

\bibitem[\protect\citeauthoryear{Jeanblanc, Yor, and Chesney}{Jeanblanc
  et~al.}{2009}]{yorbook}
Jeanblanc, M., M.~Yor, and M.~Chesney (2009).
\newblock {\em Mathematical methods for financial markets}.
\newblock Springer Verlag.

\bibitem[\protect\citeauthoryear{Levi}{Levi}{1907}]{Levi}
Levi, E.~E. (1907).
\newblock Sulle equazioni lineari totalmente ellittiche alle derivate parziali.
\newblock {\em Rend. Circ. Mat. Palermo\/}~{\em 24}, 275--317.

\bibitem[\protect\citeauthoryear{Linetsky}{Linetsky}{2006}]{linetsky2006bankruptcy}
Linetsky, V. (2006).
\newblock Pricing equity derivatives subject to bankruptcy.
\newblock {\em Mathematical Finance\/}~{\em 16\/}(2), 255--282.

\bibitem[\protect\citeauthoryear{Linetsky}{Linetsky}{2007}]{linetskybook}
Linetsky, V. (2007).
\newblock Chapter 6 {S}pectral methods in derivatives pricing.
\newblock In J.~R. Birge and V.~Linetsky (Eds.), {\em Financial Engineering},
  Volume~15 of {\em Handbooks in Operations Research and Management Science},
  pp.\  223 -- 299. Elsevier.

\bibitem[\protect\citeauthoryear{Lopez and Temme}{Lopez and
  Temme}{2002}]{lopez2002two}
Lopez, J.~L. and N.~M. Temme (2002).
\newblock Two-point {T}aylor expansions of analytic functions.
\newblock {\em Studies in Applied Mathematics\/}~{\em 109\/}(4), 297--311.

\bibitem[\protect\citeauthoryear{Lorig, Pagliarani, and Pascucci}{Lorig
  et~al.}{2013}]{lpp-1.5}
Lorig, M., S.~Pagliarani, and A.~Pascucci (2013).
\newblock Pricing approximations and error estimates for local {L}\'evy-type
  models with default.
\newblock {\em ArXiv preprint arXiv:1304.1849\/}.

\bibitem[\protect\citeauthoryear{Madan, Carr, and Chang}{Madan
  et~al.}{1998}]{madan1998variance}
Madan, D., P.~Carr, and E.~Chang (1998).
\newblock The variance gamma process and option pricing.
\newblock {\em European Finance Review\/}~{\em 2\/}(1), 79--105.

\bibitem[\protect\citeauthoryear{Mendoza-Arriaga, Carr, and
  Linetsky}{Mendoza-Arriaga et~al.}{2010}]{carr}
Mendoza-Arriaga, R., P.~Carr, and V.~Linetsky (2010).
\newblock Time-changed markov processes in unified credit-equity modeling.
\newblock {\em Mathematical Finance\/}~{\em 20}, 527--569.

\bibitem[\protect\citeauthoryear{{O}ksendal and Sulem}{{O}ksendal and
  Sulem}{2005}]{oksendal2}
{O}ksendal, B. and A.~Sulem (2005).
\newblock {\em Applied stochastic control of jump diffusions}.
\newblock Springer Verlag.

\bibitem[\protect\citeauthoryear{Pagliarani and Pascucci}{Pagliarani and
  Pascucci}{2011}]{pagliarani2011analytical}
Pagliarani, S. and A.~Pascucci (2011).
\newblock Analytical approximation of the transition density in a local
  volatility model.
\newblock {\em Central European Journal of Mathematics\/}~{\em 10(1)},
  250--270.

\bibitem[\protect\citeauthoryear{Pagliarani, Pascucci, and Riga}{Pagliarani
  et~al.}{2013}]{RigaPagliaraniPascucci}
Pagliarani, S., A.~Pascucci, and C.~Riga (2013).
\newblock Adjoint expansions in local {L}\'evy models.
\newblock {\em SIAM J. Finan. Math.\/}~{\em 4(1)}, 265--296.

\bibitem[\protect\citeauthoryear{Pascucci}{Pascucci}{2011}]{Pascucci2011}
Pascucci, A. (2011).
\newblock {\em {PDE} and martingale methods in option pricing}.
\newblock Bocconi\&Springer Series. New York: Springer-Verlag.

\end{thebibliography}

\clearpage
\begin{figure}
\centering
\begin{tabular}{cc}
$t=0.25$ & $t=1.00$ \\
\includegraphics[width=.495\textwidth,height=.325\textheight]{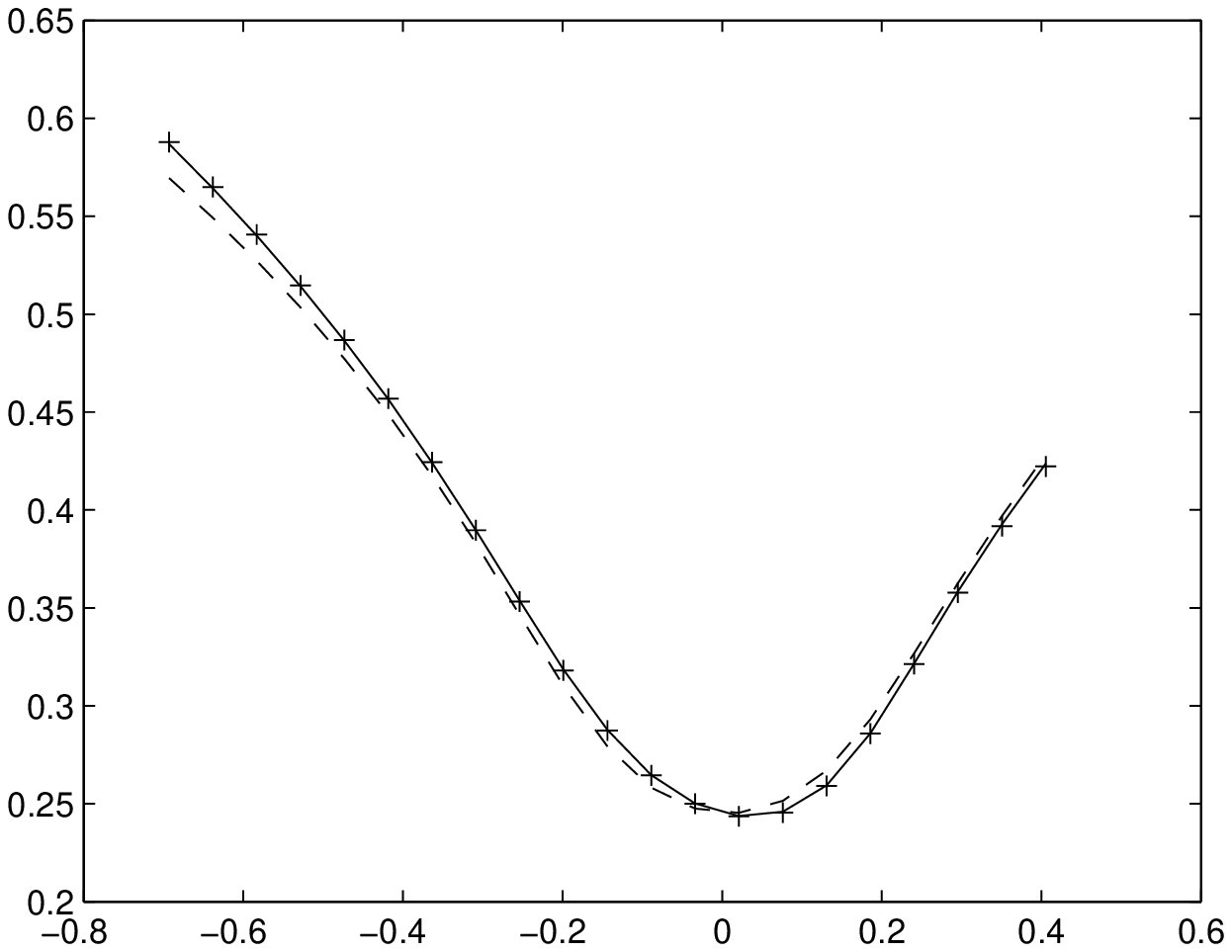} &
\includegraphics[width=.495\textwidth,height=.325\textheight]{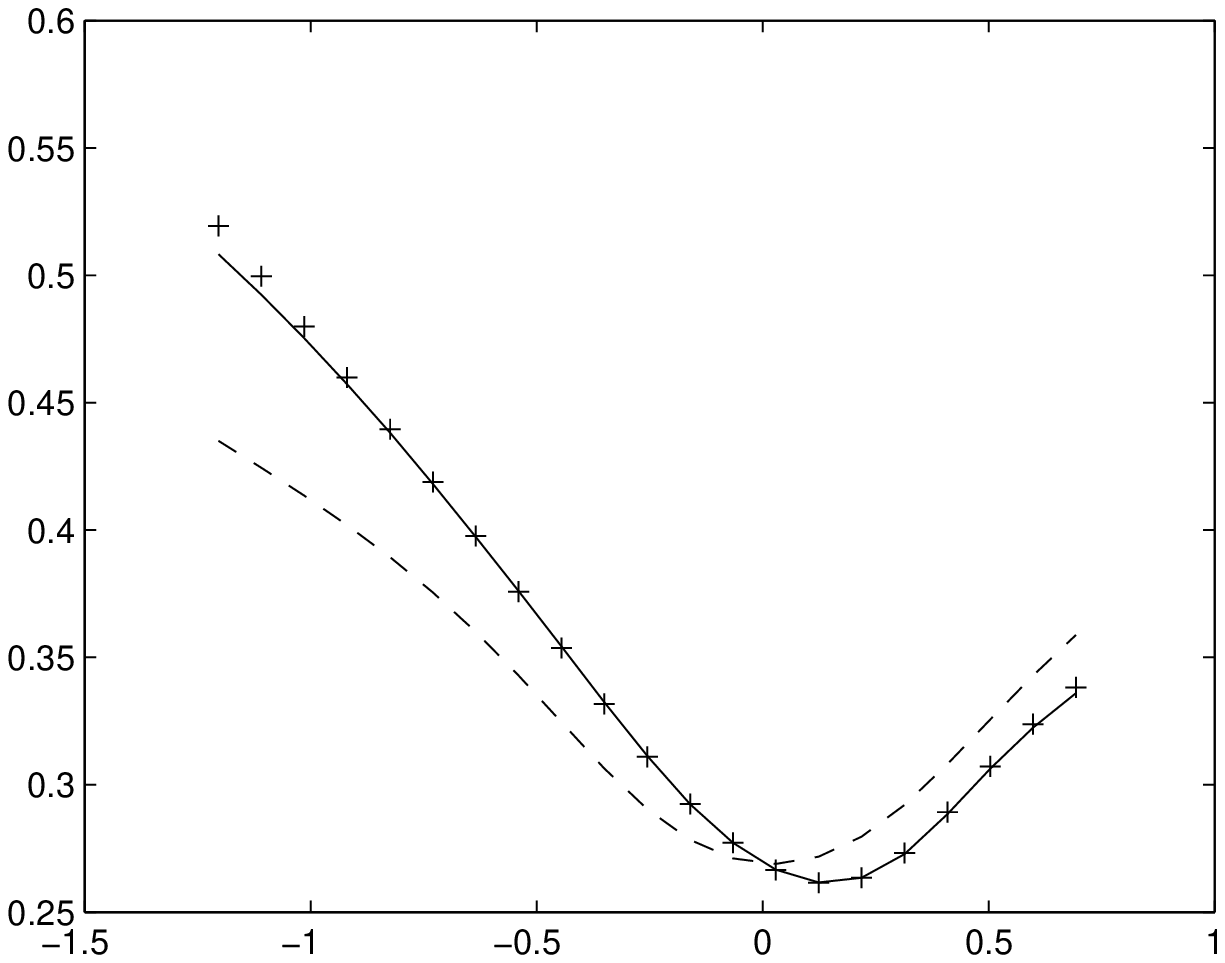} \\
$t=3.00$ & $t=5.00$\\
\includegraphics[width=.495\textwidth,height=.325\textheight]{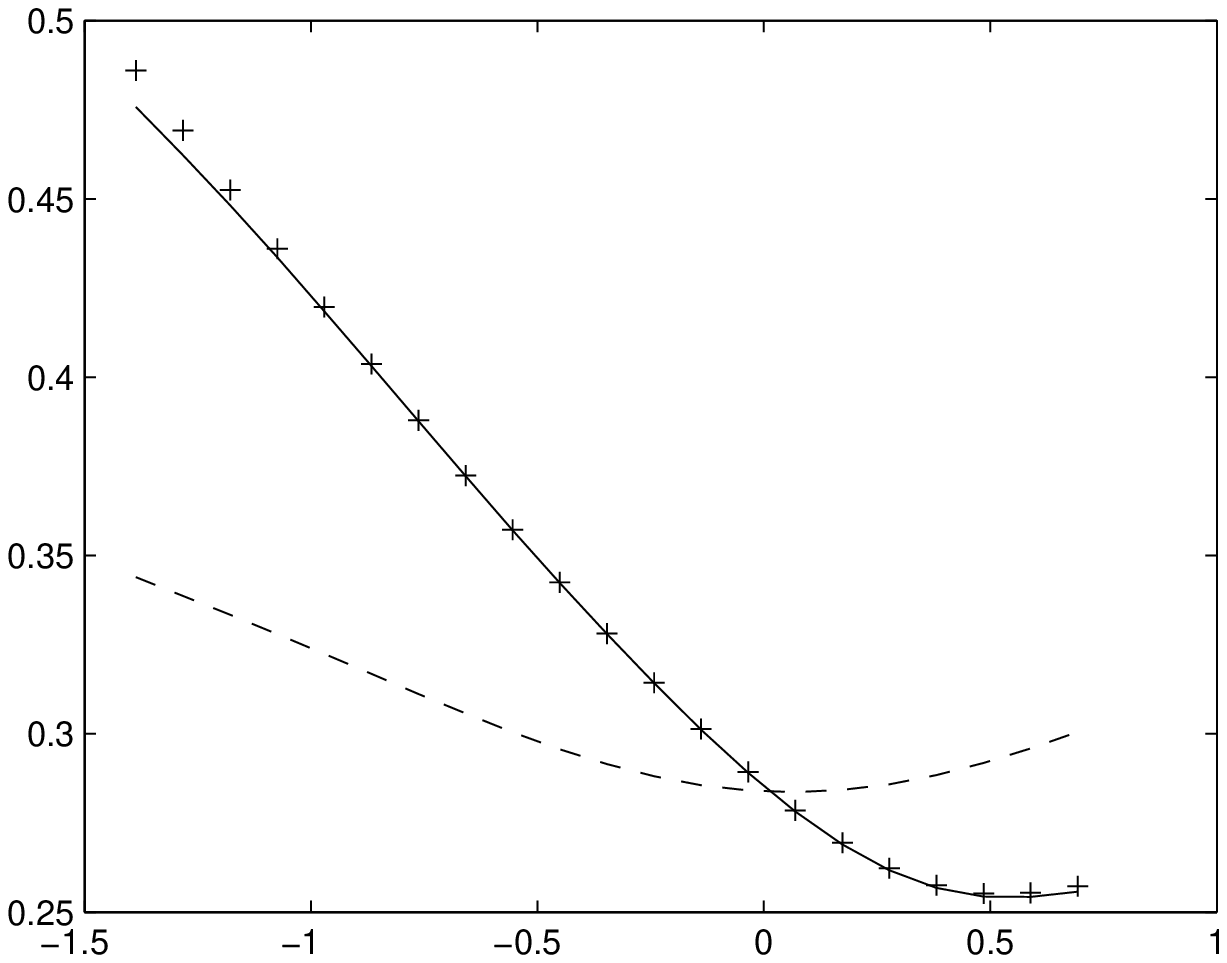} &
\includegraphics[width=.495\textwidth,height=.325\textheight]{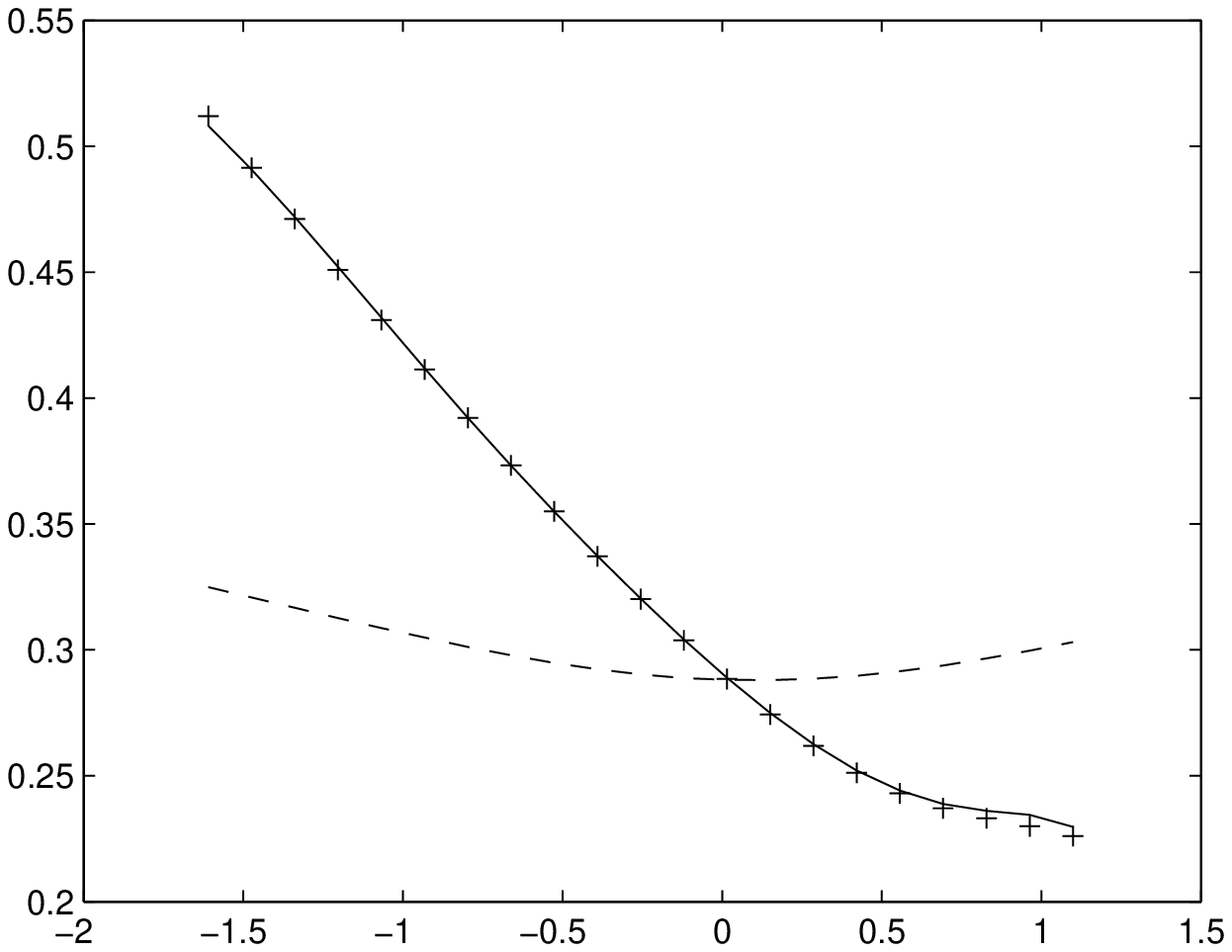}\\
\end{tabular}
\caption{Implied volatility (IV) is plotted as a function of $\log$-strike $k := \log K$ for the CEV-like model with Gaussian-type jumps of Section \ref{sec:Merton.CEV}.  The solid lines corresponds to the IV induced by $u^{(3)}(t,x)$, which is computed using the one-point Taylor expansion (see Example \ref{ex1}).  The dashed lines corresponds to the IV induced by $u^{(0)}(t,x)$ (again, using the usual one-point Taylor series expansion).  The crosses correspond to the IV induced by $u^{(MC)}(t,x)$, which is the price obtained from the Monte Carlo simulation.}
\label{fig:IV-Gauss-1}
\end{figure}

\clearpage
\begin{table*}
\centering
\begin{tabular}{c|c|cc|cc}
$t$ & $k$ & $u^{(3)}$ & $u$ MC-95\% c.i. & $\text{IV}[u^{(3)}]$ & IV MC-95\% c.i.  \\
\hline
\hline
    {}  & -0.6931  &  0.0006  &  0.0006  -  0.0007  &  0.5864  &  0.5856  -  0.5901 \\
    {}  & -0.4185  &  0.0024  &  0.0024  -  0.0025  &  0.4563  &  0.4553  -  0.4583 \\
    0.2500  & -0.1438  &  0.0111  &  0.0110  -  0.0112  &  0.2875  &  0.2865  -  0.2883 \\
    {}  &  0.1308  &  0.1511  &  0.1508  -  0.1513  &  0.2595  &  0.2573  -  0.2608 \\
    {}  &  0.4055  &  0.5028  &  0.5024  -  0.5030  &  0.4238  &  0.4152  -  0.4288 \\
\hline
    {}  & -1.2040  &  0.0009  &  0.0009  -  0.0010  &  0.5115  &  0.5176  -  0.5210 \\
    {}  & -0.7297  &  0.0046  &  0.0047  -  0.0048  &  0.4174  &  0.4178  -  0.4199 \\
    1.0000  & -0.2554  &  0.0314  &  0.0313  -  0.0316  &  0.3109  &  0.3102  -  0.3117 \\
    {}  &  0.2189  &  0.2781  &  0.2775  -  0.2784  &  0.2638  &  0.2620  -  0.2649 \\
    {}  &  0.6931  &  1.0034  &  1.0030  -  1.0041  &  0.3358  &  0.3296  -  0.3459 \\
\hline
    {}  & -1.3863  &  0.0074  &  0.0081  -  0.0083  &  0.4758  &  0.4851  -  0.4870 \\
    {} & -0.8664   & 0.0224  &  0.0224  -  0.0227  &  0.4031  &  0.4029  -  0.4045 \\
    3.0000  & -0.3466  &  0.0776  &  0.0773  -  0.0779  &  0.3280  &  0.3274  -  0.3288 \\
    {}  &  0.1733  &  0.3097  &  0.3094  -  0.3107  &  0.2690  &  0.2685  -  0.2703 \\
    {}  &  0.6931  &  1.0155  &  1.0150  -  1.0169  &  0.2558  &  0.2540  -  0.2604 \\
\hline
    {} &  -1.6094  &  0.0160  &  0.0164  -  0.0166  &  0.5082  &  0.5111  -  0.5128 \\
    {} &  -0.9324  &  0.0439  &  0.0436  -  0.0440  &  0.4118  &  0.4107  -  0.4121 \\
    5.0000 &  -0.2554  &  0.1504  &  0.1497  -  0.1507  &  0.3203  &  0.3194  -  0.3208 \\
    {} &   0.4216  &  0.6139  &  0.6123  -  0.6142  &  0.2521  &  0.2500  -  0.2524 \\
    {} &   1.0986  &  2.0050  &  2.0032  -  2.0057  &  0.2297  &  0.2163  -  0.2342 \\
\hline
\end{tabular}
\caption{Prices ($u$) and Implied volatility (IV[$u$]) as a function of time to maturity $t$ and $\log$-strike $k := \log K$ for the CEV-like model with Gaussian-type jumps of Section \ref{sec:Merton.CEV}.  The approximate price $u^{(3)}$ is computed using the (usual) one-point Taylor expansion (see Example \ref{ex1}).  For comparison, we provide the 95\% confidence intervals for prices and implied volatilities, which we obtain from the Monte Carlo simulation.}
\label{tab:IV-Gauss-1}
\end{table*}

\clearpage
\newgeometry{top=0.1cm,bottom=0.1cm}
\begin{figure}
\centering
\begin{tabular}{ccc}
$n=0$ &
\includegraphics[width=.35\textwidth,height=.15\textheight]{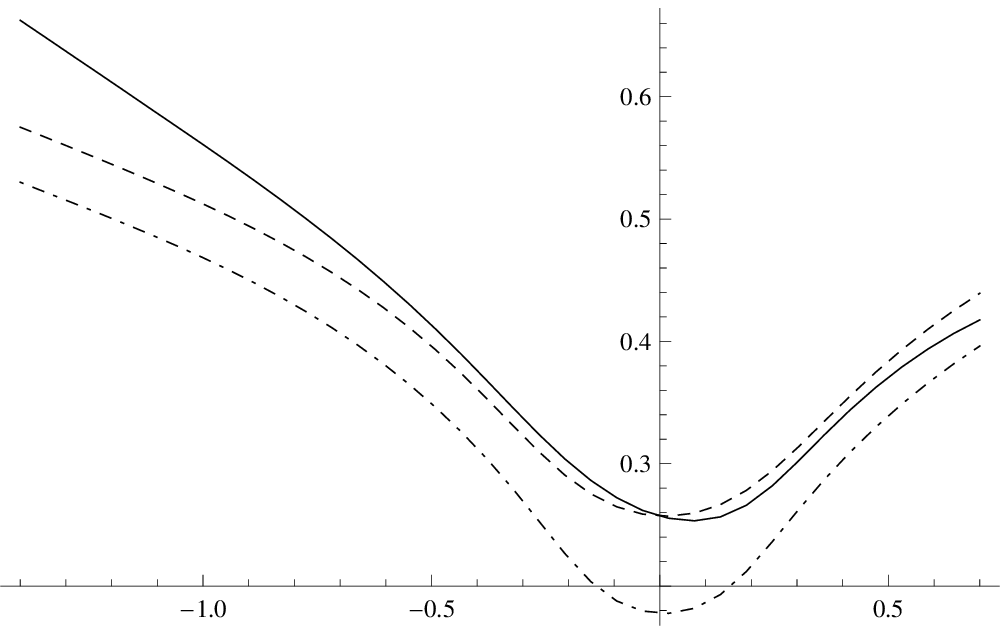} &
\includegraphics[width=.35\textwidth,height=.15\textheight]{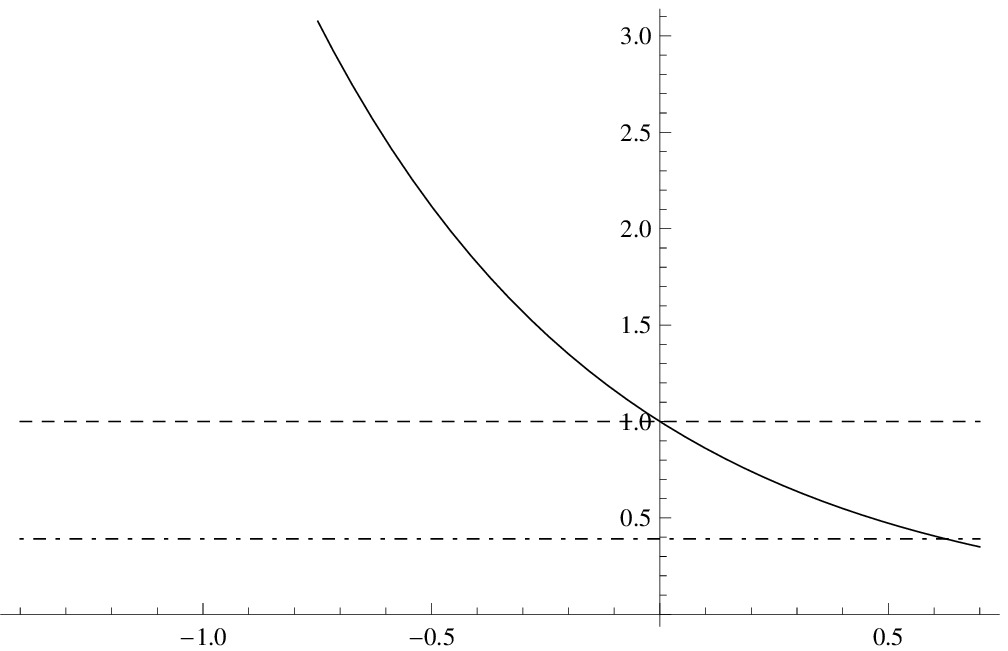} \\ \hline
$n=1$ &
\includegraphics[width=.35\textwidth,height=.15\textheight]{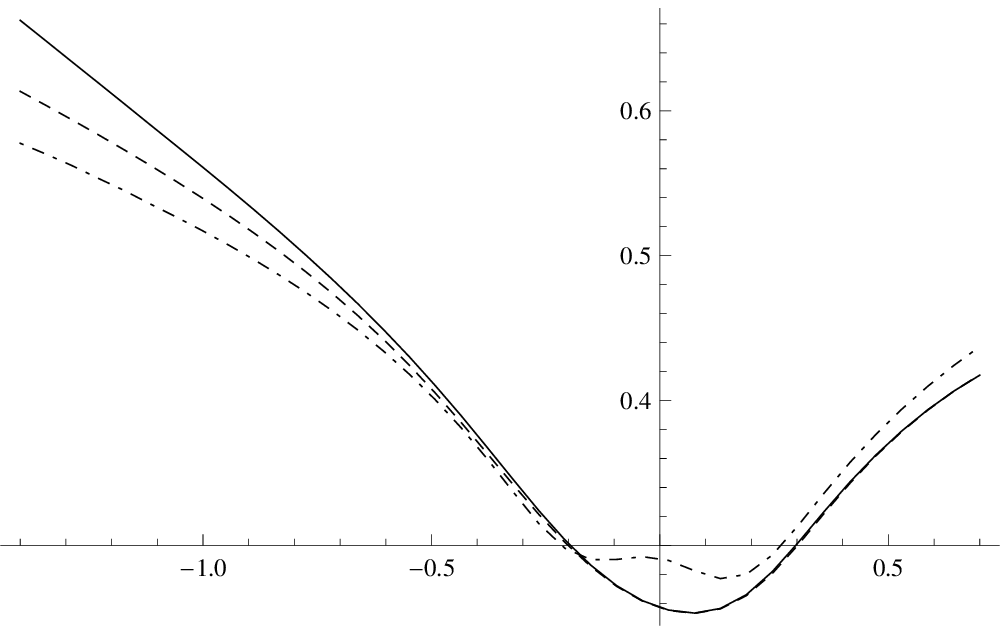} &
\includegraphics[width=.35\textwidth,height=.15\textheight]{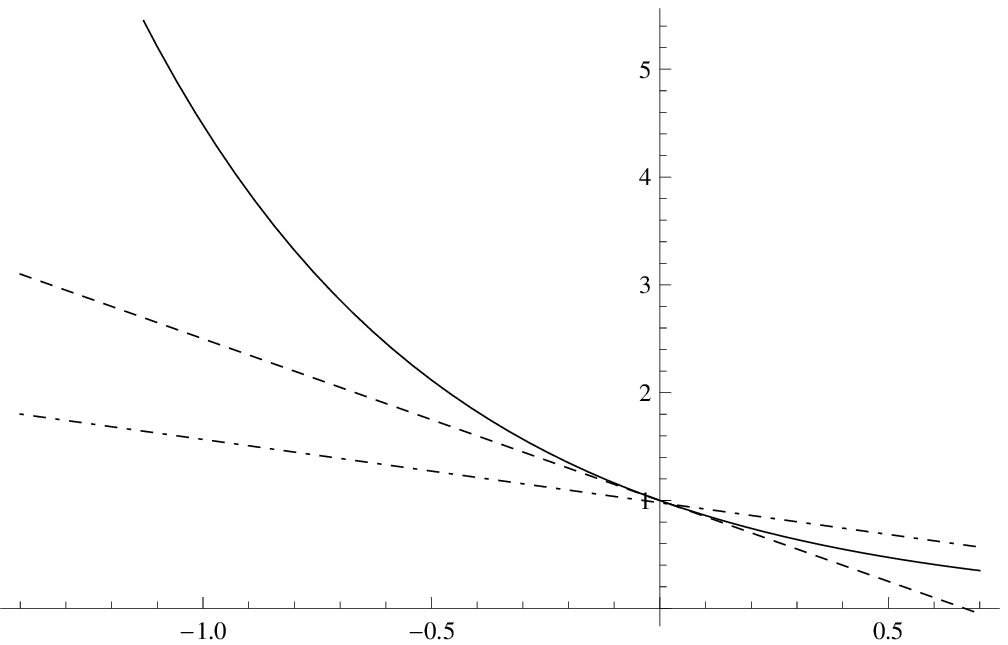} \\ \hline
$n=2$ &
\includegraphics[width=.35\textwidth,height=.15\textheight]{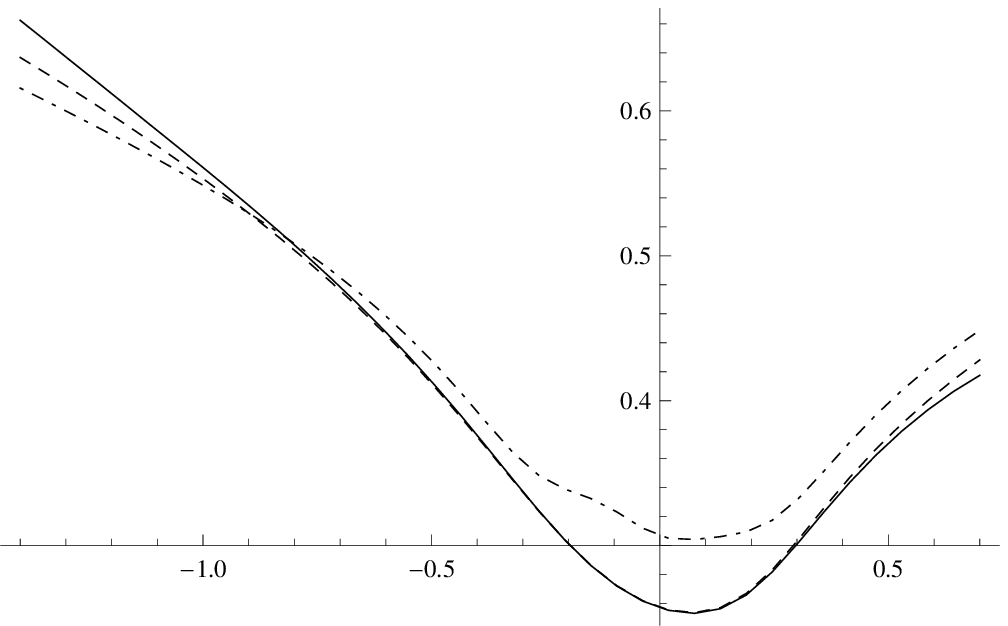} &
\includegraphics[width=.35\textwidth,height=.15\textheight]{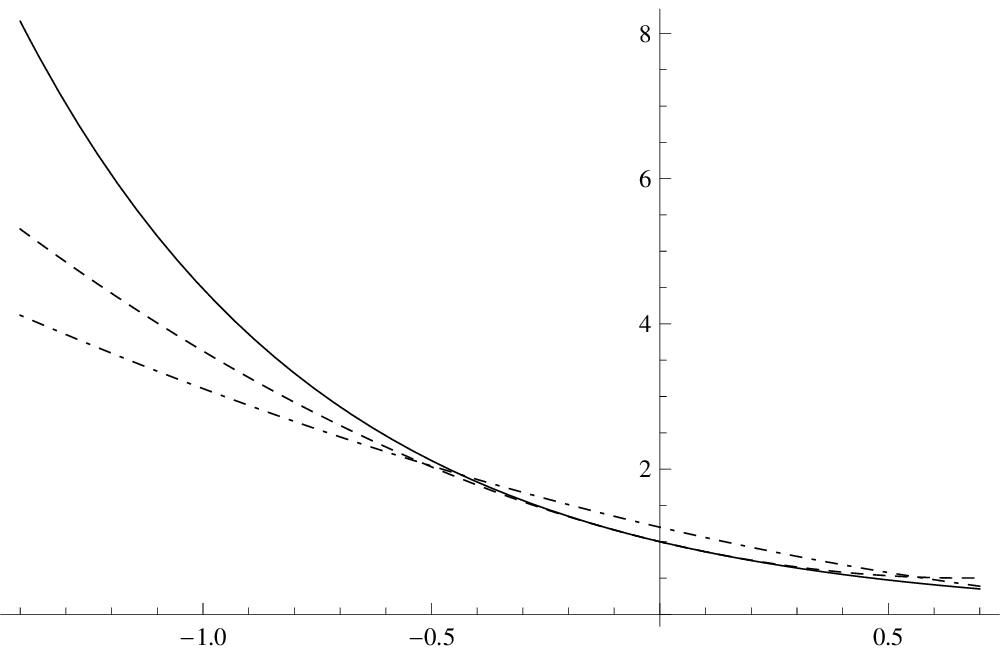} \\ \hline
$n=3$ &
\includegraphics[width=.35\textwidth,height=.15\textheight]{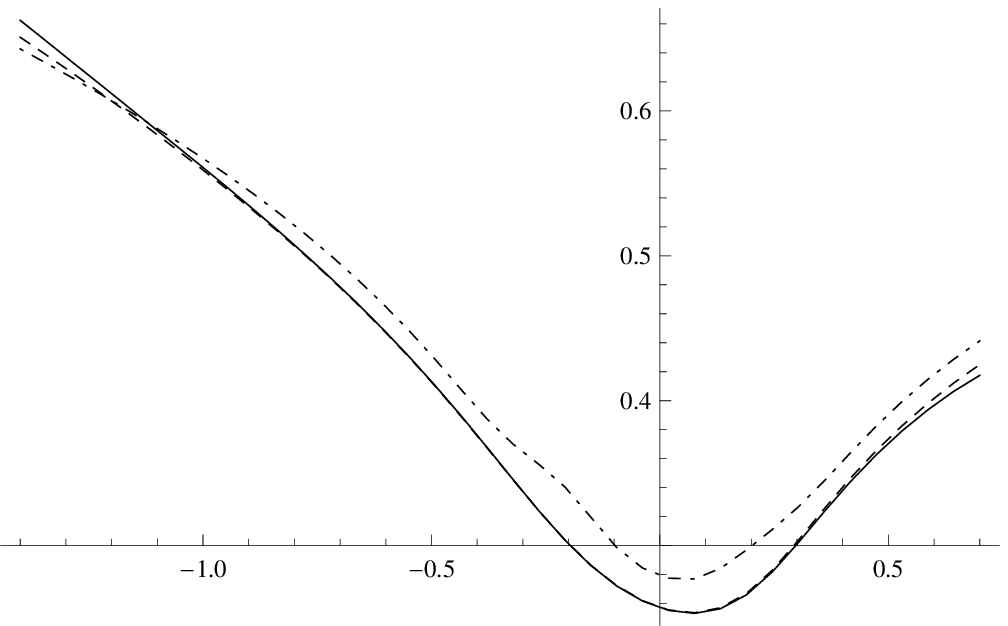} &
\includegraphics[width=.35\textwidth,height=.15\textheight]{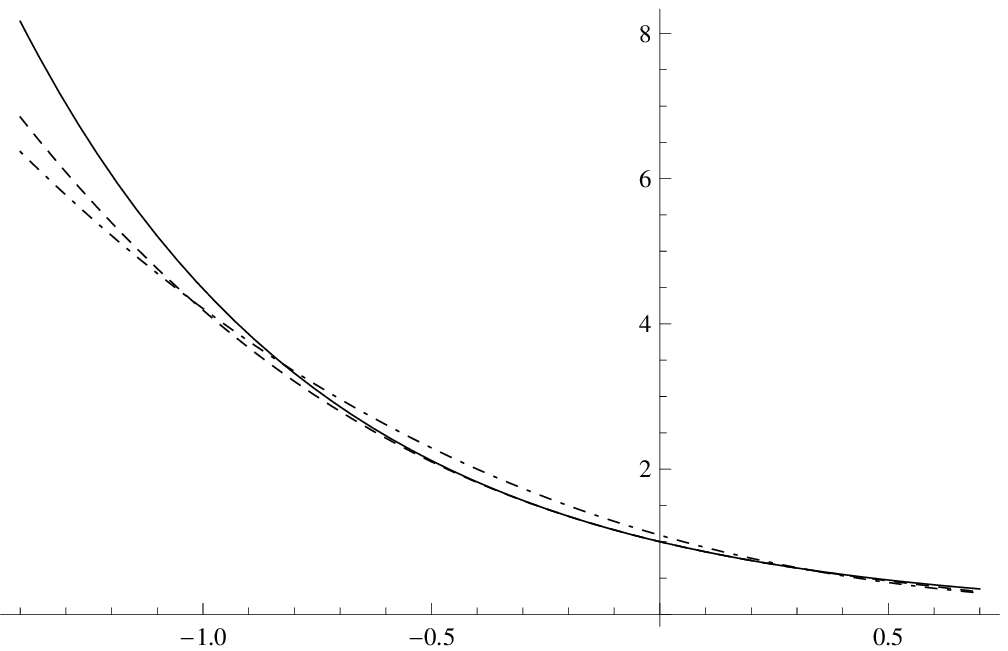} \\ \hline
$n=4$ &
\includegraphics[width=.35\textwidth,height=.15\textheight]{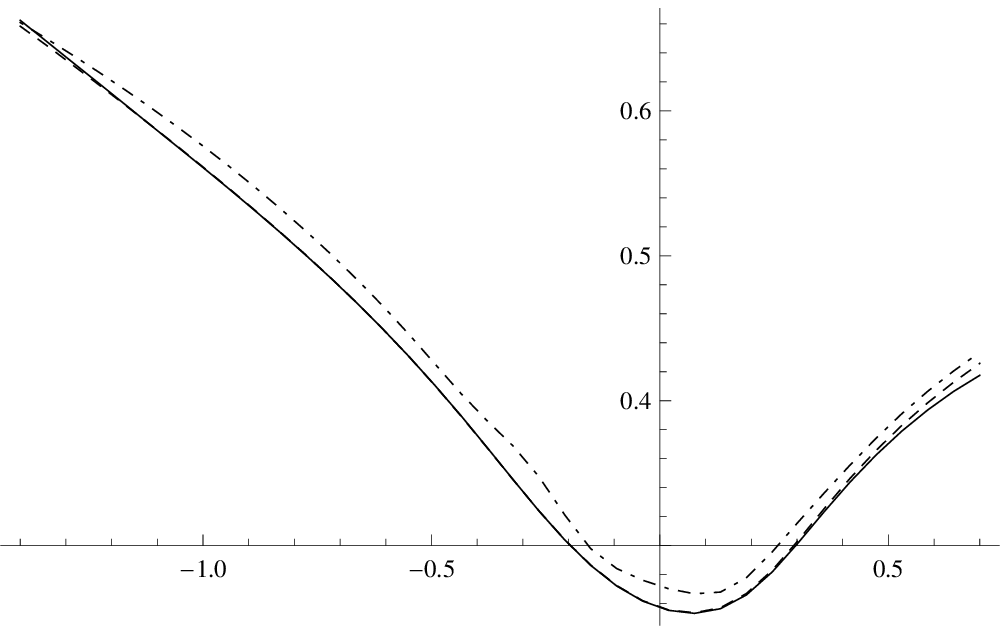} &
\includegraphics[width=.35\textwidth,height=.15\textheight]{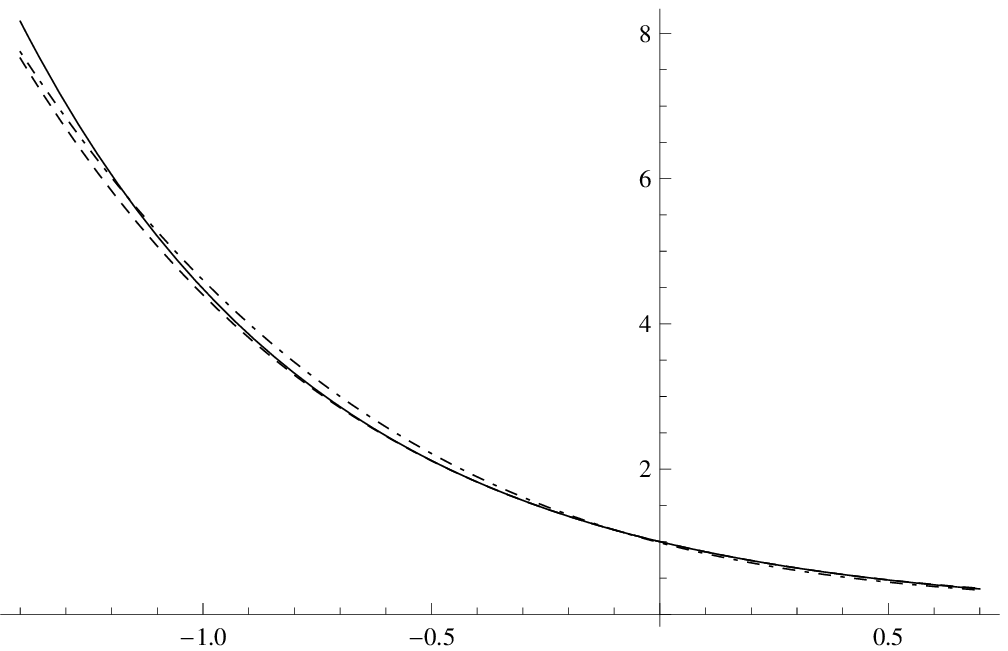} \\ \hline
\end{tabular}
\caption{Left: for the model considered in Section \ref{sec:Merton.CEV} and for a fixed maturity $t=0.5$, implied volatility is plotted as a function of $\log$-strike.  The dashed line corresponds to $\text{IV}[u^{(n)}]$ where $u^{(n)}$ is computed using Taylor series basis functions (Example \ref{ex1}).  The dot-dashed line corresponds to $\text{IV}[u^{(n)}]$ where $u^{(n)}$ is computed using Hermite polynomial basis functions (Example \ref{ex:L2}).  The solid line corresponds to $\text{IV}[u^{(MC)}]$.
Right: $f(x)=e^{2(\beta-1)x}$ (solid) and its $n$-th order Taylor series and Hermite polynomial approximations $f_\text{Taylor}^{(n)}(x)$ (dotted) and $f_\text{Hermite}^{(n)}(x)$ (dot-dashed); see equation \eqref{eq:fs}.}
\label{fig:OnePt-Hermite}
\end{figure}
\restoregeometry

\clearpage
\begin{table*}
\centering
$t = 0.25$ years\\
\begin{tabular}{c|c|cc|cc|c}
\hline
Parameters & $k = \log K$ & $u^{(3)}$ & $u$ MC-95\% c.i. & $\text{IV}[u^{(3)}]$ & IV MC-95\% c.i. & $\tau^{(3)}/\tau^{(0)}$  \\

\hline
\hline          
    $\del = 0.5432$     & -0.6000  &  0.4552  &  0.4552  -  0.4553  &  0.6849  &  0.6836  -  0.6869 & \\
    $\beta = 0.3756$  & -0.3500  &  0.3123  &  0.3122  -  0.3124  &  0.6230  &  0.6217 -   0.6242 & \\
    $\lam = 0.0518$     & -0.1000  &  0.1621  &  0.1618  -  0.1623  &  0.5704  &  0.5687  -  0.5714 & $4.9787$ \\
    $m = -0.5013$       &  0.1500  &  0.0496  &  0.0492  -  0.0500  &  0.5240  &  0.5222  -  0.5266 & \\
    $\eta = 0.3839$        &  0.4000  &  0.0059  &  0.0057  -  0.0067  &  0.4821  &  0.4787  -  0.4950 & \\
\hline          
    $\del = 0.1182$     & -0.6000  &  0.4566  &  0.4566  -  0.4567  &  0.7257  &  0.7239  -  0.7271 & \\
    $\beta = 0.9960$  & -0.3500  &  0.3137  &  0.3136  -  0.3139  &  0.6391  &  0.6378  -  0.6405 & \\
    $\lam = 0.8938$     & -0.1000  &  0.1431  &  0.1429  -  0.1434  &  0.4615  &  0.4602  -  0.4630 & $4.77419$ \\
    $m = -0.4486$       &  0.1500  &  0.0032  &  0.0030  -  0.0037  &  0.2013  &  0.1970  -  0.2073 & \\
    $\eta = 0.2619$        &  0.4000  &  0.0000  &  0.0000  -  0.0000  &  0.2510  &  0.2567  -  0.2616 & \\
\hline          
    $\del = 0.3376$     & -0.6000  &  0.4621  &  0.4619  -  0.4621  &  0.8462  &  0.8439  -  0.8478 & \\
    $\beta = 0.4805$  & -0.3500  &  0.3190  &  0.3189  -  0.3192  &  0.6949  &  0.6933  -  0.6968 & \\
    $\lam = 0.9610$     & -0.1000  &  0.1578  &  0.1575  -  0.1581  &  0.5457  &  0.5444  -  0.5476 & $4.31915$ \\
    $m = -0.2420$       &  0.1500  &  0.0451  &  0.0448  -  0.0456  &  0.4990  &  0.4974  -  0.5021 & \\
    $\eta = 0.5391$        &  0.4000  &  0.0155  &  0.0152  -  0.0162  &  0.6006  &  0.5981  -  0.6080 & \\
\hline          
    $\del = 0.2469$     & -0.6000  &  0.4592  &  0.4591  -  0.4593  &  0.7871  &  0.7857  -  0.7900 & \\
    $\beta = 0.1875$  & -0.3500  &  0.3100  &  0.3099  -  0.3102  &  0.5965  &  0.5950  -  0.5986 & \\
    $\lam = 0.4229$     & -0.1000  &  0.1341  &  0.1338  -  0.1343  &  0.4083  &  0.4069  -  0.4096 & $4.46032$ \\
    $m = -0.2823$       &  0.1500  &  0.0306  &  0.0302  -  0.0309  &  0.4149  &  0.4126  -  0.4168 & \\
    $\eta = 0.7564$        &  0.4000  &  0.0176  &  0.0171  -  0.0179  &  0.6213  &  0.6171  -  0.6244 & \\
\hline
\end{tabular}
\caption{After selecting model parameters randomly, we compute call prices ($u$) for the CEV-like
model with Gaussian-type jumps discussed in Section \ref{sec:Merton.CEV}.  For each strike, the
approximate call price $u^{(3)}$ is computed using the (usual) one-point Taylor expansion (see
Example \ref{ex1}) as well as by Monte Carlo simulation.  The obtained prices, as well as the
associated implied volatilities (IV[$u$]) are displayed above.  Note that,
the approximate price $u^{(3)}$ (and corresponding implied volatility) consistently falls within
the 95\% confidence interval obtained from the Monte Carlo simulation. We denote by $\tau^{(n)}$
the total time it takes to compute the $n$-th order approximation of option prices $u^{(n)}$ at the five strikes
displayed in the table.  Because total computation time depends on processor speed, in the last
column, we give the ratio $\tau^{(3)}/\tau^{(0)}$.  Note that $\tau^{(0)}$ is a useful benchmark,
as it corresponds to the total time it takes to compute the five call in an Exponential L\'evy
setting (i.e., option prices with no local dependence) using standard Fourier techniques.}
\label{tab:IV-Gauss-2}
\end{table*}

\clearpage
\begin{table*}
\centering
$t = 1.00$ years\\
\begin{tabular}{c|c|cc|cc|c}
\hline
Parameters & $k = \log K$ & $u^{(3)}$ & $u$ MC-95\% c.i. & $\text{IV}[u^{(3)}]$ & IV MC-95\% c.i. & $\tau^{(3)}/\tau^{(0)}$  \\
\hline
\hline          
    $\del = 0.5806$     & -1.0000  &  0.6487  &  0.6486  -  0.6488  &  0.7306  &  0.7294  -  0.7319 & \\
    $\beta = 0.5829$  & -0.6000  &  0.5001  &  0.5000  -  0.5004  &  0.6719  &  0.6711  -  0.6734 & \\
    $\lam = 0.0367$     & -0.2000  &  0.3220  &  0.3216  -  0.3224  &  0.6167  &  0.6157  -  0.6182 & $4.97872$ \\
    $m = -0.6622$       &  0.2000  &  0.1512  &  0.1507  -  0.1520  &  0.5649  &  0.5636  -  0.5671 & \\
    $\eta = 0.2984$        &  0.6000  &  0.0413  &  0.0408  -  0.0428  &  0.5166  &  0.5145  -  0.5219 & \\
\hline          
    $\del = 0.3921$     & -1.0000  &  0.6556  &  0.6555  -  0.6561  &  0.8022  &  0.8014  -  0.8075 & \\
    $\beta = 0.1271$  & -0.6000  &  0.5012  &  0.5011  -  0.5018  &  0.6779  &  0.6772  -  0.6809 & \\
    $\lam = 0.4176$     & -0.2000  &  0.3052  &  0.3051  -  0.3060  &  0.5655  &  0.5651  -  0.5678 & $4.54839$ \\
    $m = -0.1661$       &  0.2000  &  0.1188  &  0.1184  -  0.1198  &  0.4832  &  0.4822  -  0.4858 & \\
    $\eta = 0.5823$        &  0.6000  &  0.0299  &  0.0296  -  0.0315  &  0.4708  &  0.4694  -  0.4772 & \\
\hline          
    $\del = 0.5803$     & -1.0000  &  0.6679  &  0.6677  -  0.6681  &  0.9122  &  0.9108  -  0.9140 & \\
    $\beta = 0.2426$  & -0.6000  &  0.5237  &  0.5236  -  0.5243  &  0.7916  &  0.7913  -  0.7943 & \\
    $\lam = 0.5926$     & -0.2000  &  0.3436  &  0.3431  -  0.3441  &  0.6830  &  0.6814  -  0.6845 & $4.3125$ \\
    $m = -0.0877$       &  0.2000  &  0.1592  &  0.1581  -  0.1596  &  0.5851  &  0.5823  -  0.5862 & \\
    $\eta = 0.3236$        &  0.6000  &  0.0373  &  0.0358  -  0.0379  &  0.5009  &  0.4949  -  0.5033 & \\
\hline          
    $\del = 0.3096$     & -1.0000  &  0.6323  &  0.6323  -  0.6324  &  0.36740  &  0.3680  -  0.3708 & \\
    $\beta = 0.6417$  & -0.6000  &  0.4554  &  0.4553  -  0.4554  &  0.34493  &  0.3442  -  0.3456 & \\
    $\lam = 0.3806$     & -0.2000  &  0.2283  &  0.2281  -  0.2284  &  0.32159  &  0.3208  -  0.3221 & $4.9257$ \\
    $m = -0.02824$      &  0.2000  &  0.0495  &  0.0491  -  0.0500  &  0.29930  &  0.2980  -  0.3006 & \\
    $\eta = 0.0122$        &  0.6000  &  0.0021  &  0.0015  -  0.0027  &  0.27807  &  0.2655  -  0.2888 & \\
\hline
\end{tabular}
\caption{After selecting model parameters randomly, we compute call prices ($u$) for the CEV-like
model with Gaussian-type jumps discussed in Section \ref{sec:Merton.CEV}.  For each strike, the
approximate call price $u^{(3)}$ is computed using the (usual) one-point Taylor expansion (see
Example \ref{ex1}) as well as by Monte Carlo simulation.  The obtained prices, as well as the
associated implied volatilities (IV[$u$]) are displayed above.  Note that,
the approximate price $u^{(3)}$ (and corresponding implied volatility) consistently falls within
the 95\% confidence interval obtained from the Monte Carlo simulation. We denote by $\tau^{(n)}$
the total time it takes to compute the $n$-th order approximation of option prices $u^{(n)}$ at the five strikes
displayed in the table.  Because total computation time depends on processor speed, in the last
column, we give the ratio $\tau^{(3)}/\tau^{(0)}$.  Note that $\tau^{(0)}$ is a useful benchmark,
as it corresponds to the total time it takes to compute the five call in an Exponential L\'evy
setting (i.e., option prices with no local dependence) using standard Fourier techniques.}
\label{tab:IV-Gauss-3}
\end{table*}


\clearpage
\begin{figure}
\centering
\begin{tabular}{cc}
$t=0.50$ & $t=1.00$ \\
\includegraphics[width=.495\textwidth,height=.325\textheight]{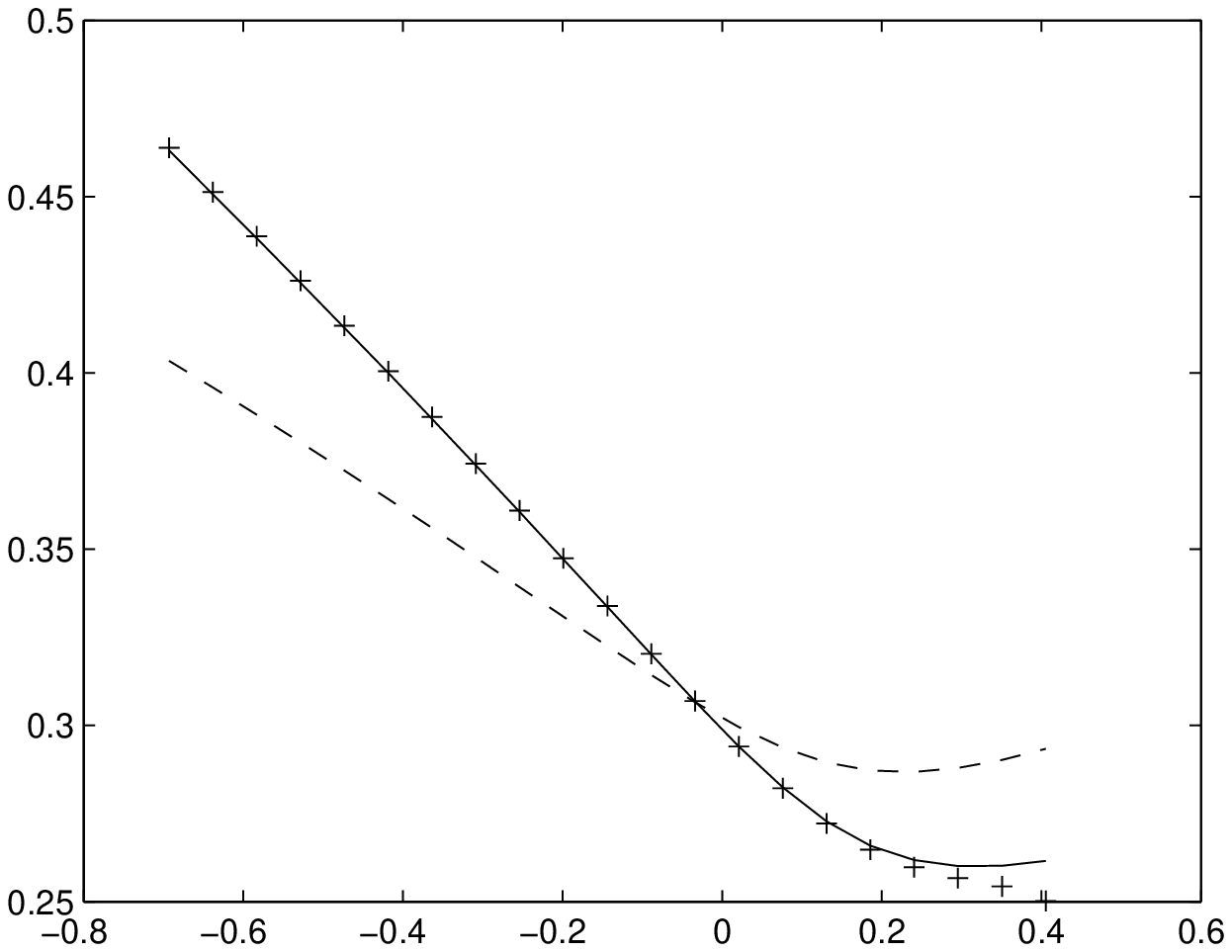} &
\includegraphics[width=.495\textwidth,height=.325\textheight]{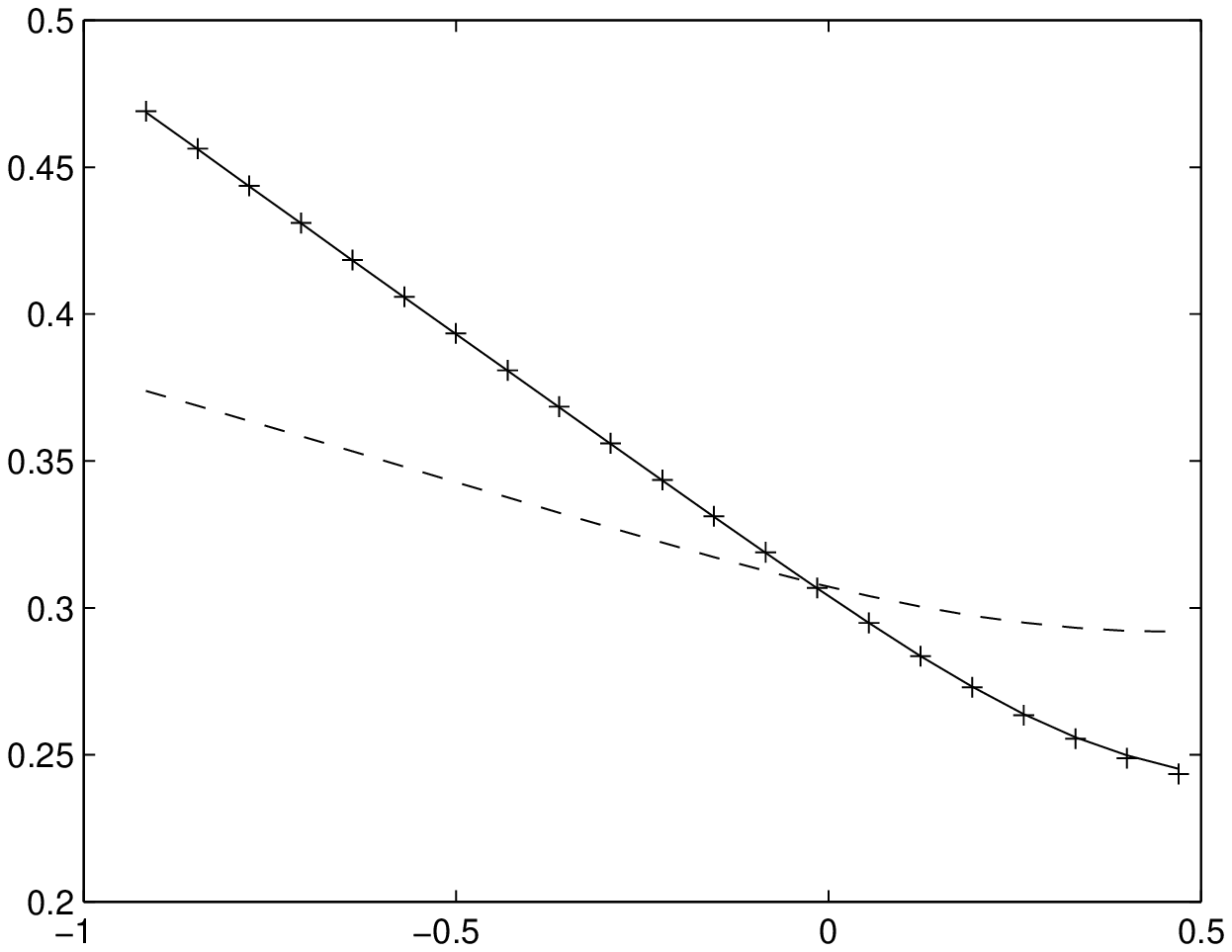} \\
\end{tabular}
\caption{Implied volatility (IV) is plotted as a function of $\log$-strike $k := \log K$ for the CEV-like model with Variance Gamma-type jumps of Section \ref{sec:VG.CEV}.  The solid lines corresponds to the IV induced by $u^{(2)}(t,x)$, which is computed using the two-point Taylor expansion (see Example \ref{ex:two-pt}).  The dashed lines corresponds to the IV induced by $u^{(0)}(t,x)$ (again, computed using the two-point Taylor series expansion).
The crosses correspond to the IV induced by $u^{(MC)}(t,x)$, which is the price obtained from the Monte Carlo simulation.}
\label{fig:IV-VG-MonteCarlo}
\end{figure}

\begin{table*}
\centering
\begin{tabular}{c|c|cc|cc}
$t$ & $k$ & $u^{(2)}$ & $u$ MC-95\% c.i. & $\text{IV}[u^{(2)}]$ & IV MC-95\% c.i.  \\
\hline
\hline
    {}  & -0.6931  &  0.0014  &  0.0014  -  0.0015  &  0.4631  &  0.4624  -  0.4652 \\
    {}  & -0.4185  &  0.0070  &  0.0070  -  0.0071  &  0.4000  &  0.3995  -  0.4014 \\
    0.5000  & -0.1438  &  0.0363  &  0.0362  -  0.0365  &  0.3336  &  0.3331  -  0.3346 \\
    {}  &  0.1308  &  0.1702  &  0.1697  -  0.1704  &  0.2727  &  0.2707  -  0.2736 \\
    {}  &  0.4055  &  0.5011  &  0.5004  -  0.5012  &  0.2615  &  0.2291  -  0.2646 \\
\hline
    {}  & -0.9163  &  0.0028  &  0.0027  -  0.0028  &  0.4687  &  0.4678  -  0.4702 \\
    {}  & -0.5697  &  0.0109  &  0.0109  -  0.0110  &  0.4057  &  0.4050  -  0.4068 \\
    1.0000  & -0.2231  &  0.0473  &  0.0472  -  0.0476  &  0.3434  &  0.3428  -  0.3444 \\
    {}  &  0.1234  &  0.1970  &  0.1965  -  0.1974  &  0.2836  &  0.2825  -  0.2847 \\
    {}  &  0.4700  &  0.6033  &  0.6025  -  0.6037  &  0.2452  &  0.2355  -  0.2506 \\
\hline
\end{tabular}
\caption{Prices ($u$), Implied volatilities (IV[$u$]) and the corresponding confidence intervals from Figure \ref{fig:IV-VG-MonteCarlo}.}
\label{tab:IV-VG-MonteCarlo1}
\end{table*}

\end{document}